\documentclass[english]{amsart}
\usepackage{amsmath,amssymb,amscd,amsfonts,mathrsfs}
\usepackage[alphabetic]{amsrefs}
\usepackage{tabularx}
\usepackage{latexsym}
\usepackage[all,cmtip,line]{xy}
\usepackage{enumitem}
\usepackage[only,llbracket,rrbracket]{stmaryrd}
\usepackage{todonotes}
\usepackage{comment}

\DeclareSymbolFont{euletters}{U}{eur}{m}{n}
\DeclareSymbolFont{eufrakletters}{U}{euf}{m}{n}
\DeclareFontFamily{U}{wncy}{}
    \DeclareFontShape{U}{wncy}{m}{n}{<->wncyr10}{}
    \DeclareSymbolFont{mcy}{U}{wncy}{m}{n}
    \DeclareMathSymbol{\Sha}{\mathord}{mcy}{"58}

\newcommand{\id}{\text {\rm id}}

\newcommand{\Hom}{\text {\rm Hom}}

\newcommand{\Ext}{\text{\rm Ext}}

\newcommand{\Res}{\text{\rm Res}}

\newcommand{\Lie}{\text {\rm Lie\,}}
\newcommand{\Spec}{{\operatorname{Spec\,}}}
\DeclareMathOperator{\Sp}{\text{\rm Sp}}

\newcommand{\an}{{\text{\rm an}}}

\newcommand{\hyp}{\mathbb{H}}

\newcommand{\ord}{{\text{\rm ord}}}

\newcommand{\ed}{{\text{\rm ed}}}

\renewcommand{\Im}{\text {\rm Im }}

\newcommand{\PSL}{\text{\rm PSL}}
\newcommand{\PGL}{\text{ \rm PGL}}
\newcommand{\SL}{\text{\rm SL}}
\newcommand{\GL}{\text{\rm GL}}

\newcommand{\SO}{\text{\rm SO}}

\newcommand{\PU}{\text{\rm PU}}
\newcommand{\U}{\text{\rm U}}
\newcommand{\PSp}{\text{\rm PSp}}

\newcommand{\ev}{ev}

\renewcommand{\th}{^{\text{th}}}

\newcommand{\Gal}{\text{\rm Gal}}

\newcommand{\A}{\mathcal A}

\renewcommand{\AA}{\mathbb A}

\newcommand{\B}{\mathcal B}
\newcommand{\C}{\mathcal C}
\newcommand{\CC}{\mathbb C}

\newcommand{\E}{\mathcal E}

\DeclareMathSymbol{\fk}\mathord{eufrakletters}{"6B}

\newcommand{\F}{\mathbb F}

\renewcommand{\H}{\mathcal H}

\DeclareMathSymbol{\ei}\mathord{euletters}{"69}
\DeclareMathSymbol{\etau}\mathord{euletters}{"1C}
\DeclareMathSymbol{\eiota}\mathord{euletters}{"13}

\DeclareMathSymbol{\eK}\mathord{euletters}{"4B}

\renewcommand{\L}{\mathcal L}

\newcommand{\M}{\mathcal M}

\newcommand{\gm}{\mathfrak m}

\renewcommand{\O}{\mathcal O}

\newcommand{\Q}{\mathbb Q}

\newcommand{\RR}{\mathbb R}

\newcommand{\UU}{\mathscr U}

\newcommand{\Z}{\mathbb Z}
\newcommand{\ZZ}{\mathcal Z}

\newcommand{\et}{\text{\rm\'et}}

\newcommand{ \iso} {\overset \sim \longrightarrow}

\newcommand{\Aut}{\text {\rm Aut}}

\newcommand{\lps}{[\![}
\newcommand{\rps}{]\!]}
\newcommand{\Pb}{\mathbb{P}}

\newcommand{\Rb}{\mathbb{R}}
\DeclareMathOperator{\RD}{RD}

\DeclareMathOperator{\Stab}{Stab}

\DeclareMathOperator{\Mon}{Mon}

\numberwithin{equation}{section}

\newtheorem{ithm}{Theorem}

\newtheorem{iprop}[ithm]{Proposition}
\newtheorem{iconj}[ithm]{Conjecture}

\newtheorem{athm}{Theorem}[section]

\newtheorem{acor}[athm]{Corollary}

\newtheorem{alemma}{Lemma}[athm]
\newtheorem{thm}{Theorem}[subsection]
\newtheorem{corollary}[thm]{Corollary}
\newtheorem{cor}[thm]{Corollary}

\newtheorem{lemma}[thm]{Lemma}

\newtheorem{prop}[thm]{Proposition}

 \theoremstyle{definition}
\newtheorem{defn}[thm]{Definition} 

\newtheorem{example}[thm]{Example}
\newtheorem{para}[thm]{\bf}

\newtheorem{question}[thm]{Question}{\bf}
{\bf}
\newcommand{\mar}[1]{\marginpar{\tiny #1}}
\theoremstyle{definition}

\newtheorem{innercustomthm}{Theorem}
\newenvironment{customthm}[1]
{\renewcommand\theinnercustomthm{#1}\innercustomthm}
{\endinnercustomthm}

\theoremstyle{definition}

\newtheorem{rem}[thm]{\bf Remark}
\newtheorem{remark}[thm]{\bf Remark}

\begin{document}

\title{Modular functions and resolvent problems}


\author[Benson Farb, Mark Kisin and Jesse Wolfson]{Benson Farb, Mark Kisin and Jesse Wolfson \\
\  \\
\ 
with an appendix by Nate Harman}
\address{ Department of Mathematics, University of Chicago}
\email{farb@math.uchicago.edu}
\address{ Department of Mathematics, Harvard }
\email{kisin@math.harvard.edu}
\address{ Department of Mathematics, University of California-Irvine}
\email{wolfson@uci.edu}
\address{ Department of Mathematics, University of Chicago}
\email{nateharman1234@gmail.com}


\thanks{The authors are partially supported by NSF grants DMS-1811772 and Jump Trading Mathlab Research Fund (BF), DMS-1601054 (MK) and
DMS-1811846 (JW)}


\begin{abstract} 

The link between modular functions and algebraic functions was a driving force behind the 19th century study of both.  Examples include the solutions by Hermite and Klein of the quintic via elliptic modular functions and the general sextic via level $2$ hyperelliptic functions.  This paper aims to apply modern arithmetic techniques to the circle of ``resolvent problems'' formulated and pursued by Klein, Hilbert and others.  As one example, we 
prove that the essential dimension at $p=2$ for the symmetric groups $S_n$ is equal to the essential dimension at $2$ of certain $S_n$-coverings defined using moduli spaces of principally polarized abelian varieties.  Our proofs use the deformation theory of abelian varieties in characteristic $p$, specifically Serre-Tate theory, as well as a family of remarkable mod $2$ symplectic $S_n$-representations constructed by Jordan.  As shown in an appendix by Nate Harman, the properties we need for such representations exist only in the $p=2$ case.

In the second half of this paper we introduce the notion of $\E$-versality as a kind of generalization of Kummer theory, and we prove that many congruence covers are $\E$-versal.  We use these $\E$-versality result to deduce the equivalence of Hilbert's 13th Problem (and related conjectures) with problems about congruence covers.

%
%

\end{abstract}



\maketitle
\tableofcontents

\section{Introduction}\label{sec:intro}
The link between modular functions and algebraic functions was a driving force behind the 19th century development of both. Examples include the solutions by Hermite and Klein of the quintic via elliptic modular functions, degree 7 and 8 equations with Galois group $\PSL_2(\F_7)$ via the level the level $7$ modular curve, the general sextic via level $2$ hyperelliptic functions, the $27$ lines on smooth cubic surfaces via level $3$, dimension $2$ abelian functions, and the $28$ bitangents on a smooth quartic via level $2$, dimension $3$ abelian functions.\footnote{See e.g. \cites{Klein8,KleinIcos,KleinLetter,Burkhardt1,Burkhardt2,Burkhardt3,KleinFricke,FrickeKlein,Fricke}, as well as \cites{KleinCW,KleinCW3}.}  With the Nazi destruction of the G\"ottingen research community this connection was largely abandoned, and the study of algebraic functions and resolvent problems, as pioneered by Klein, Hilbert and others, fell into relative obscurity.  The purpose of this paper to reconsider the link between modular functions and classical resolvent problems.  We do this from a modern viewpoint, using arithmetic techniques.

\medskip
\noindent
{\bf Essential dimension at \boldmath$p$ of modular functions.}  To fix ideas we work over $\CC$.  Recall that an {\em algebraic function} is a finite correspondence 
$X\dashrightarrow^{\!\!\!\!\!\!\!\!\!\!1:n} \Pb^1$; that is, a rational function $f: \tilde{X}\dashrightarrow \Pb^1$ on some (finite, possibly branched) cover $\tilde{X}\to X.$  \footnote{When the functions are understood, we denote an algebraic function simply by the cover $\tilde{X}\to X$.}  A fundamental example is the {\em general degree $n$ polynomial}, equivalently the cover 
\[
	\M_{0,n}\to \M_{0,n}/S_n,
\]
where $\M_{0,n}$ denotes the moduli space of $n$ distinct marked points in $\Pb^1$.  When $X$ is a locally symmetric variety $f$ is called a {\em modular function}.  A basic example is the cover  $\A_{g,N}\to A_g$ where $A_g$ is the (coarse) moduli space of principally polarized $g$-dimensional abelian varieties and $\A_{g,N}$ is the moduli of pairs $(A,\B)$ with $A\in\A_g$ and $\B$ a symplectic basis for $H_1(A;\Z/N\Z)$.  

The relationship between modular functions and the solutions of the general degree $n$ polynomial motivated Klein \cite{KleinIcos,KleinLetter}, 
Kronecker  \cite{Kronecker} and others to ask about the intrinsic complexity of these algebraic functions, as measured by the number of variables to which they can be reduced after a rational change of variables. In modern terms (as defined by Buhler-Reichstein, see e.g. \cite{ReICM}), the  {\em essential dimension}  $ed(\tilde{X}/X)\leq \dim(X)$ of an algebraic function is the smallest $d\geq 1$ so that $\tilde{X}\to X$ is the birational pullback of a cover $\tilde{Y}\to Y$ of $d$-dimensional varieties.    

One can also allow, in addition to rational changes of coordinates, the adjunction of radicals or other algebraic functions.  This is done by specifying a class $\E$ of covers 
under which  $\tilde{X}\to X$ can be pulled back before taking $\ed$ of the resulting cover.  This gives 
the essential dimension $\ed(\tilde{X}/X;\E)$ relative to the class $\E$ of ``accesory irrationalities''.  For example, if one fixes a prime $p$ and pulls back by covers of degree prime to $p,$ one obtains the notion of {\em essential dimension at $p$}, denoted $\ed(\tilde X/X; p)$ (see e.g. \cite{ReiYo}).  The idea of accessory irrationality was  central to the approaches of Klein and Hilbert to solving equations.  We axiomatize this notion in Definition~\ref{definition:accesory} below and explore its consequences in Section~\ref{s:equiv}.

The general degree $n$ polynomial is universal for covers with Galois group $S_n$, even allowing prime-to-$p$ accessory 
irrationalities; that is, for all $p\geq 2$ and for $\ed(S_n;p)$ defined as the maximum of 
$\ed(\tilde{X}/X;p)$ for all $S_n$-covers $\tilde{X}\to X$, we have:
\[\ed(\M_{0,n}/\M_{0,n};p)=\ed(S_n;p).\]

With the many examples relating the general degree $n$ polynomial to modular functions, it is natural to ask if the same ``maximal complexity property'' holds for modular functions.  Our first result states that for $p=2$ this is indeed the case.  To explain this, for a subgroup $G\subset\Sp_{2g}(\Z/N\Z)$ set $\A_{g,G}:=\A_{g,N}/G.$

\begin{ithm}\label{it:main} Let $n\geq 2,$ $g=\lceil \frac{n}{2}\rceil -1,$  and let $N \geq 3$ be odd. 
 There exists an embedding $S_n \subset \Sp_{2g}(\F_2) \subset \Sp_{2g}(\Z/2N\Z)$ such that 
$$ \ed(\A_{g,2N}/\A_{g,S_n}; 2) = \lfloor n/2\rfloor =  \ed(S_n; 2). $$
            \end{ithm}          
            
We remark that what we actually prove is the first equality. The second equality then 
follows from a result of Meyer-Reichstein \cite[Corollary 4.2]{MR}. In particular, one sees from their result that 
$\ed(S_n;p)$ takes its maximal value for $p=2,$ so this case is, in some sense, the most interesting. 

One ingredient in the proof of Theorem~\ref{it:main} comes from the link between binary forms and hyperelliptic functions; specifically,  Jordan proved that the monodromy of the $2$-torsion points on the universal hyperelliptic Jacobian gives a mod $2$ symplectic $S_n$-representation.   These remarkable representations were rediscovered and studied by Dickson \cite{Di}  in 1908.  We deduce Theorem~\ref{it:main} by applying the following general result to these representations.
\begin{ithm}\label{ithm:mainthm} Let $G$ be a finite group, and $G\to \Sp_{2g}(\F_p)$ a representation. If $U \subset \Sp_{2g}$ 
is the unipotent of a Siegel parabolic then 
	\[
		\ed(\A_{g,pN}/\A_{g,G};p)\ge \dim_{\F_p} G \cap U(\F_p)
	\]
\end{ithm}

Theorem~\ref{ithm:mainthm} is of most interest for those $G$ which admit a symplectic representation with $\dim_{\F_p} G \cap U(\F_p) = \ed(G;p),$ 
where $\ed(G;p)$ is the essential dimension at $p$ of a {\em versal} branched cover with group $G$ (see Definition~\ref{d:evers} below). 
For $G = S_n,$ a result of Harman (Theorems \ref{charnot2} and \ref{char2}) 
says that this is possible only for $p=2,$ and only using the particular mod $2$ symplectic representation of Jordan/Dickson!   We also show that for $G$ the $\F_q$-points of a split semisimple group of classical type, there is a symplectic representation of $G$ for which the lower bound in Theorem \ref{ithm:mainthm} is either equal or nearly equal to the maximal rank of an elementary abelian $p$-group in $G.$  The only near-misses occur for odd orthogonal groups.  Note however, that this rank is in general less than $\ed(G;p).$ 

\medskip
\noindent
{\bf \boldmath$\E$-versal modular functions.} Kummer theory gives that for each $d\geq 2$ the cover $\Pb^1\to \Pb^1/(\Z/d\Z)$ has the following universal property: any $\Z/d\Z$ cover $\tilde{X}\to X$ is pulled back from it.  It follows that  $\ed(\tilde{X}/X;p)=1$ for any such $\tilde{X}\to X$.  
Klein's {\it Normalformsatz} states that, while the icosahedral cover $\Pb^1\to \Pb^1/A_5$ is not universal in the above sense (indeed $\ed(\M_{0,5}\to \M_{0,5}/A_5)=2$), there exists a $\Z/2\Z$  accessory irrationality 
\[
\begin{array}{ccc}
\tilde{Y}&\to&\tilde{X}\\
\downarrow&&\downarrow\\
Y&\to&X
\end{array}
\]
such that $\tilde{Y}\to Y$ is a pullback of $\Pb^1\to \Pb^1/A_5$. This nonabelian version of Kummer's theorem is a kind of classification of actions of $A_5$ on all varieties.  We say in this case that $\Pb^1\to \Pb^1/A_5$ is {\em $\E$-versal} with respect to any collection $\E$ of covers containing $\Z/2\Z$ covers. Note that this cover is modular; indeed it is equivariantly birational to the cover $\hyp^2/\Gamma_2(5)\to\hyp^2/\SL_2(\Z)$, where $\hyp^2$ is the hyperbolic plane and $\Gamma_2(5)$ is the level $5$ congruence subgroup of $\SL_2(\Z)$; here we are using the natural isomorphism $\PSL_2(\F_5)\cong A_5$.    

In \S\ref{s:equiv} we axiomatize the idea of $\E$-versality and we give a number of examples (most classically known) of congruence covers that are $\E$-versal for various groups $G$. One sample result on $\E$-versality is the following (see \S\ref{subsection:versal} for terminology).   For $\Gamma<\SL_2(\Rb)\times\SL_2(\Rb)$ a lattice, let  $M_\Gamma:(\hyp^2\times\hyp^2)/\Gamma$; these are complex-algebraic varieties called {\em Hilbert modular surfaces}. 

\begin{iprop}
 For $\E$ any class of accessory irrationalities containing all quadratic and cubic covers and composites thereof, the Hilbert modular surface
			\[
				M_{\widetilde{\SL_2(\Z[\frac{1+\sqrt{5}}{2}],3)}}\to M_{\SL_2(\Z[\frac{1+\sqrt{5}}{2}])}
			\]
			is $\E$-versal for $A_6$, where $\widetilde{\SL_2(\Z[\frac{1+\sqrt{5}}{2}],3)}$ denotes the kernel of the map \[\SL_2(\Z[\frac{1+\sqrt{5}}{2}])\to\PGL_2(\F_9)\cong A_6.\]  In particular, Hilbert's Sextic Conjecture is equivalent to the statement that the resolvent degree of this cover equals $2$.
	\end{iprop}
	
The connection between these $\E$-versality results with the first part of this paper is that $\E$-versal $G$-covers always maximize $\ed(\tilde{X}/X;\E)$ over all $G$-covers $\tilde{X}\to X$.   In \S\ref{s:equiv} we apply such $\E$-versality results to exhibit further the close relationship between modular functions and roots of polynomials. Specifically, Hilbert's 13th Problem, and his Sextic and Octic Conjectures (see \S\ref{s:equiv} for their exact statements) are phrased in terms of the resolvent degree of the degree $6,7$ and $8$ polynomials.  The {\em resolvent degree} $\RD(\tilde{X}/X)$ is the smallest $d$ such that $\tilde{X}\to X$ is covered by a composite of covers, each of essential dimension $\leq d$ (see e.g. \cite{AS,Bra,FW}). 
 Applying various $\E$-versality results, we deduce in \S\ref{s:equiv}  the equivalence of each of Hilbert's conjectures with a conjecture about the resolvent degree of a specific modular cover. Similarly, we show that such a modular reformulation is possible not only for general polynomials of low degree, but also for each of the algebraic functions considered by Klein and his school \cites{KleinFirst,KleinLetter,KleinHist,Fricke}. 

\medskip
\noindent
{\bf Methods. }The proof of Theorem \ref{ithm:mainthm} uses a refinement of the results of \cite{FKW}, which is explained in \S 1. 
In {\em loc.~cit}, we used Serre-Tate theory to give lower bounds on the essential at $p$ for the coverings $\A_{g,pN} \rightarrow \A_{g,N},$ 
when restricted to (some) subvarieties $\ZZ \subset \A_{g,N}.$ Here we drop the assumption that $\ZZ$ is a subvariety and allow certain maps 
$\ZZ \rightarrow \A_{g,N}$ (cf.\ Proposition~\ref{prop:essdimAg}).   In particular, we can apply the resulting estimate to $\ZZ = \A_{g,G}$ for 
$G$ a subgroup of $\Sp_{2g}(\F_p),$ which yields the lower bound for $\ed(\A_{g,pN}/ \A_{g,G};p)$ in Theorem \ref{ithm:mainthm}.

One may compare the bounds given by Theorem \ref{ithm:mainthm} to those obtained in \cite[\S 4]{FKW} for 
certain finite simple groups of Lie type. 
The bound in the case of odd orthogonal groups in {\em loc.~cit} is weaker than the one given here because of the restriction on 
the signature of Hermitian symmetric domains associated to odd orthogonal groups. On the other hand the coverings we consider here correspond to rather more exotic congruence subgroups than those of {\em loc.~cit}.

\bigskip
\noindent
{\bf Acknowledgements.} {The authors would like to thank Robert Guralnick for pointing out a mistake in an earlier version of this paper. 
We also thank Igor Dolgachev, Bert van Geemen, Bruce Hunt,  Aaron Landesman, Zinovy Reichstein and Ron Solomon for helpful correspondence. }

\section{Moduli of Abelian varieties}\label{sec:moduli}
\numberwithin{equation}{thm}
\subsection{Extension classes}\label{subsec:extns}
\begin{para}
Fix a prime $p,$ and let $V$ be a complete discrete valuation ring of characteristic $0,$
with perfect residue field $k$ of characteristic $p,$ and a uniformizer $\pi \in V.$
Let $A = V\lps x_1, \dots, x_n \rps$ be a power series ring over $V.$
We denote by $\gm_A \subset A$ the maximal ideal, and $\bar \gm_A = \gm_A/\pi A.$
and set $\tilde X = \Spec A,$ and $X = \Spec A[1/p].$ We will denote by $k[\epsilon] = k[X]/X^2$ the dual numbers over $k.$

Recall \cite[3.1.2]{FKW} that there is a commutative diagram
$$\xymatrix{
A^\times/(A^\times)^p \ar[r]^\sim\ar[d] & \Ext^1_{\tilde X}(\Z/p\Z, \mu_p) \ar[d] \\
A[1/p]^\times/(A[1/p]^\times)^p \ar[r]^{\phantom{ttt}\sim} & \Ext^1_X(\Z/p\Z, \mu_p)
 }$$
where the terms on the right are extensions as $\Z/p\Z$-sheaves.
The vertical maps are injective, and the extensions in the image of the map on the right are 
called {\it syntomic}.  There is also a map \cite[3.1.5]{FKW}
$$ \theta_A: \Ext^1_{\tilde X}(\Z/p\Z, \mu_p) \iso A^\times/(A^\times)^p \rightarrow \bar\gm_A/\bar\gm_A^2.$$
which sends a class represented by a function $f \in 1 + \gm_A$ to $f-1.$

\end{para}

\begin{lemma}\label{lem:nondeg} Let $\UU \subset  \Ext^1_{\tilde X}(\Z/p\Z, \mu_p)$ be an $\F_p$-subspace of dimension $\leq n.$
Suppose that for every map $h:A \rightarrow k[\epsilon]$ the image of $\UU$ under the induced map
\begin{equation}\label{eqn:nondeg}
\Ext^1_{\tilde X}(\Z/p\Z, \mu_p) \rightarrow \Ext^1_{\Spec k[\epsilon]}(\Z/p\Z, \mu_p)
\end{equation}
is nontrivial. Then the map
\begin{equation}\label{eqn:nondegII}
\theta_A: \UU\otimes_{\F_p} k \rightarrow \bar \gm_A/\bar \gm_A^2
\end{equation}
is an isomorphism; in particular $\dim_{\F_p}\UU = n.$
\end{lemma}
\begin{proof}
Since the image of $\UU$ under \ref{eqn:nondeg} is nontrivial, the composite
$$ \theta_A: \UU\otimes_{\F_p} k \rightarrow \bar \gm_A/\bar \gm_A^2 \rightarrow \epsilon\cdot k$$
is nontrivial for every $h.$ This implies that \ref{eqn:nondegII} is surjective, and since $\dim_{\F_p}\UU \leq n$
it is injective, and $\dim_{\F_p}\UU = n.$
\end{proof}

\begin{para}\label{para:extns} We call a subspace $\UU \subset  \Ext^1_{\tilde X}(\Z/p\Z, \mu_p)$ satisfying the conditions of Lemma~\ref{lem:nondeg} {\em nondegenerate},
and we fix such a subspace. Now assume that $V$ contains a primitive $p\th$ root of unity, and fix a geometric point $\bar x$ of $X.$
Then 
$$ \Ext^1_X(\Z/p\Z, \mu_p) \iso H^1(X, \mu_p) = H^1_{\et}(X, \mu_p) = \Hom(\pi_1(X,\bar x), \mu_p).$$
If $\UU' \subset \UU$ is a subspace,  denote by $X(\UU') \rightarrow X$ the finite \'etale cover corresponding to $\UU'.$
That is, $X(\UU')$ is the cover corresponding to the intersection of all the elements of $\Hom(\pi_1(X,\bar x), \mu_p)$ that are images of elements of $\UU'.$ We let $\tilde X' = \Spec A(\UU')$ denote the normalization of $\tilde X$ in $X(\UU').$
\end{para}

\begin{lemma}\label{lem:span} For any $\UU' \subset \UU$ the ring $A(\UU')$ is a power series ring over $V.$  Further, 

\begin{equation}\label{eqn:span}
\dim_k \, \Im\!\!(\bar\gm_A/\bar\gm_A^2 \rightarrow \bar \gm_{A(\UU')}/ \gm^2_{A(\UU')}) = \dim_{\F_p}(\UU/\UU').
\end{equation}
\end{lemma}
\begin{proof}
 Let $f_1,\dots, f_r \in A^\times$ be elements with $1-f_i \in \gm_A,$ and such that the images of $f_1,\dots, f_r$
in $\Ext^1_{\tilde X}(\Z/p\Z, \mu_p)$ form an $\F_p$-basis for $\UU'.$
By definition,  $X(\UU') = \Spec A[1/p](\sqrt[p]{f_1}, \dots, \sqrt[p]{f_r}).$ To prove the first claim,
it suffices to show that
$$ A(\sqrt[p]{f_1}, \dots, \sqrt[p]{f_r}) = A[z_1,\dots, z_r]/(z_i^p-f_i)$$
is a power series ring over $V.$  
Since $\UU$ is nondegenerate, the images of $f_1,\dots, f_r$ are $k$-linearly independent in $\gm_A/\gm_A^2.$
Hence, after a change of coordinates, we can assume that $A \iso V\lps x_1,\dots, x_n\rps $ with
$x_i = f_i-1$ for $i=1,\dots r.$ Then we have
$$ A[z_1,\dots, z_r]/(z_i^p-f_i) \iso V\lps z_1-1, \dots, z_r-1, x_{r+1}, \dots, x_n \rps.$$

This also shows \ref{eqn:span}, as both sides are equal to $n-r.$
\end{proof}

\subsection{Monodromy on the ordinary locus}
\begin{para}
Fix an integer $g \geq 1$, a prime $p\geq 2$, and a positive integer $N \geq 2$ coprime to $p.$
Consider the ring $\Z[\zeta_N][1/N],$ where $\zeta_N$ is a primitive $N^{\text{\rm th}}$ root of $1.$
Denote by $\A_{g,N}$ the $\Z[\zeta_N][1/N]$-scheme which is the coarse moduli space of principally polarized abelian schemes
 $A$ of dimension $g$ equipped with a basis of $A[N]$ that is symplectic with respect to the Weil pairing defined by $\zeta_N.$
When $N \geq 3,$ this is a fine moduli space which is smooth over $\Z[\zeta_N][1/N].$ For a $\Z[\zeta_N][1/N]$-algebra $B,$ denote by
$\A_{g,N/B}$ the base change of $\A_{g,N}$ to $B.$ If no confusion is likely to result, we sometimes denote this base change simply by $\A_{g,N}.$

From now on, unless stated otherwise, we assume that $N \geq 3$ and we let $\A \rightarrow \A_{g,N}$ be the universal abelian scheme. The $p$-torsion subgroup $\A[p] \subset \A$ is a finite flat group scheme over $\A_{g,N}$ which is \'etale over $\Z[\zeta_N][1/Np].$ Let $x \in \A_{g,N}$ be a point with residue field $\kappa(x)$ of characteristic $p,$ and $\A_x$ the corresponding abelian variety over $\kappa(x).$

The set of points $x$ such that $\A_x$ is ordinary  is an open subscheme $\A_{g,N}^\ord \subset \A_{g,N}\otimes \F_p$.
We denote by $\widehat \A_{g,N}^\ord$ the formal completion of $\A_{g,N}$ along
$\A_{g,N}^\ord.$ We denote by $\A_{g,N}^{\ord,\an}$ the ``generic fibre'' of $\widehat \A_{g,N}^\ord$ as a $p$-adic analytic space.\footnote{The reader may think of any version of the theory of $p$-adic analytic spaces they prefer (Tate, Raynaud, Berkovich, or H\"uber's adic spaces), as this will have no bearing on our arguments.}

Denote by $k$ an algebraically closed perfect field of characteristic $p,$ and let $K/W[1/p]$ be a finite extension with ring of integers $\O_K$ and uniformizer $\pi.$  Assume that $K$ is equipped with a choice of primitive $N^{\text{\rm th}}$ root of $1,$ $\zeta_N \in K,$
so that we may consider all the objects introduced above over $\O_K.$  Let $\bar K/K$ be an algebraic closure.  

\end{para}
\begin{prop}\label{prop:monodromy} Fix a geometric point $x \in \widehat \A_{g,N}^{\ord,\an}(\bar K)$ and denote by $\bar x \in \A_{g,N}^{\ord}$ its reduction.The covering $\A_{g,pN} \rightarrow  \A_{g,N}$ corresponds to a surjective representation
\begin{equation}\label{eqn:monodromy}
\pi_1(\A_{g,N}, x) \rightarrow \Sp_{2g}(\F_p).
\end{equation}
\begin{enumerate}
\item There exists a Siegel parabolic $P \subset  \Sp_{2g}/_{\F_p}$ with unipotent radical $U,$ such that
\ref{eqn:monodromy} induces a surjective representation
\begin{equation}\label{eqn:monodromyII}
\pi_1(\widehat \A_{g,N}^{\ord,\an},x) \rightarrow P(\F_p).
\end{equation}
\item Let $A = \widehat \O_{\A_{g,N},\bar x}$ be the completion of the local ring at $\bar x.$
Then (\ref{eqn:monodromy}) induces a surjective representation
\begin{equation}\label{eqn:monodromyIII}
\pi_1(\Spec A[1/p], x) \rightarrow U(\F_p).
\end{equation}
\end{enumerate}
\end{prop}
\begin{proof} The first claim is well known. Indeed, the existence of the Weil pairing on $\A[p]$ implies that
$\A_{g,pN}$ corresponds to a symplectic representation. A comparison with the topological fundamental group shows
that the image of the {\em geometric} fundamental group $\pi_1(\A_{g,N}\otimes_K\bar K, x)$ is $\Sp_{2g}(\F_p),$
so the representation is surjective.

Now recall, that a {\em Siegel parabolic}  is the stabilizer of a maximal isotropic subspace in the underlying vector space of a symplectic representation.
Equivalently it is a parabolic with abelian unipotent radical. All such parabolics are conjugate. Over $\widehat \A_{g,N}^{\ord}$ the finite flat group scheme $\A[p]$ is an extension
\begin{equation}\label{eqn:ptorsionextension} 0 \rightarrow \A[p]^m \rightarrow \A[p] \rightarrow \A[p]^{\et} \rightarrow 0
\end{equation}
of an \'etale by a multiplicative group scheme, where \'etale locally $\A[p]^{\et} \iso (\Z/p\Z)^g$ and $\A[p]^m \iso \mu_p^g.$
The Weil pairing induces a map of group schemes
$$ \A[p] \times \A[p] \rightarrow \mu_p.$$
which identifies $\A[p]$ with its Cartier dual, and induces an isomorphism $\A[p]^m$ with the Cartier dual of $\A[p]^{\et}.$
In particular, this shows that $\A[p]^m_x \subset \A[p]_x$ corresponds to a maximal isotropic subspace under the Weil pairing.
This defines a Siegel parabolic such that (\ref{eqn:monodromy}) maps $ \pi_1(\widehat \A_{g,N}^{\ord,\an},x)$ into $P(\F_p).$
By \cite[Prop. 7.2]{FC}  the image of the composite
$$ \pi_1(\widehat \A_{g,N}^{\ord,\an},x) \rightarrow P(\F_p) \rightarrow (P/U)(\F_p)$$
is surjective. Hence it suffices to prove (2).

For this, we adopt the notation of \ref{subsec:extns} applied with $A$ as in (2). Since we are assuming $k$ is algebraically closed,
over $A,$ the group schemes $\A[p]^{\et}$ and $\A[p]^m$ are isomorphic to $(\Z/p\Z)^g$ and $\mu_p^g$ respectively.
In particular, the map (\ref{eqn:monodromyIII}) factors through $U(\F_p).$
Let $\UU \subset \Ext^1_X(\Z/p\Z, \mu_p)$ be the span of the $g^2$ syntomic extension classes defining the extension
(\ref{eqn:ptorsionextension}).
Note that $U(\F_p)$ is an elementary abelian $p$-group of rank $n = \dim_{\F_p} U = \dim \A_g = \binom{g+1}{2}.$
Any $\F_p$-linear map $s:U(\F_p) \rightarrow \F_p$ induces a representation
$$\pi_1(\Spec A[1/p], x) \rightarrow \mu_p(\bar K) \iso \F_p$$
(choosing $p\th$ root of unity), and hence a class in
$$c(s) \in \Ext^1_X(\Z/p\Z, \mu_p)\iso H^1(X, \F_p).$$
The subspace $\UU$ is the span of all the classes $c(s).$ This shows $\dim \UU \leq n,$
with equality only if  (\ref{eqn:monodromyIII}) is surjective.
However, by \cite[3.2.2]{FKW}, one sees that $\UU$ satisfies the conditions of Lemma \ref{lem:nondeg},
so that $\dim \UU = n,$ which completes the proof of the lemma.
\end{proof}

\begin{cor}\label{cor:monodromy} With the notation above, $\Hom_{\F_p}(U(\F_p),\F_p)$ is naturally identified with a
nondegenerate subspace $\UU \subset \Ext^1_X(\Z/p\Z, \mu_p).$
\end{cor}
\begin{proof} The proof of the Proposition~\ref{prop:monodromy} shows that there is a natural map
$$ \Hom_{\F_p}(U(\F_p),\F_p) \rightarrow \Ext^1_X(\Z/p\Z, \mu_p) $$
whose image $\UU$ is a nondegenerate subspace of dimension $n = \dim_{\F_p} U.$
\end{proof}

\subsection{Essential dimension}
\begin{para} We refer the reader to \cite[\S 2]{FKW} for the definitions and facts we will need about essential dimension and essential dimension at $p$. We remind the reader that for $K$ a field and $Y \rightarrow X$ a finite \'etale map of finite type $K$-schemes, $\ed(Y/X; p)$ denotes the essential dimension at $p$ of $Y_{\bar K} \rightarrow X_{\bar K},$ where $\bar K$ is an algebraic closure of $K.$

\end{para}
\begin{para} We continue to use the notation introduced above. In particular $A = \widehat \O_{\A_{g,N},\bar x}$ denotes the complete local ring which is a power series ring over $\O_K$ in $n = \binom{g+1}{2}$ variables.
\end{para}

\begin{lemma}\label{lem:basicedestimate} Let $g: A \rightarrow B$ and $f: C \rightarrow B $ be maps of power series rings over $\O_K,$
with $f$ a flat map. Suppose there exists a finite \'etale covering $Y' \rightarrow \Spec C[1/p]$ and an isomorphism of
\'etale coverings $\varepsilon: f^* Y' \iso g^*\A[p]$ over $\Spec B[1/p].$
Then
$$ \Im\!\!(\bar\gm_A/\bar\gm_A^2 \rightarrow \bar\gm_B/\bar\gm_B^2) \subset \Im\!\!(\bar\gm_C/\bar\gm_C^2 \rightarrow \bar\gm_B/\bar\gm_B^2).$$
In particular,
$$ \dim_k \bar\gm_C/\bar\gm_C^2 \geq  \dim_k \,  \Im\!\!(\bar\gm_A/\bar\gm_A^2 \rightarrow \bar\gm_B/\bar\gm_B^2). $$
\end{lemma}
\begin{proof} By \cite[2.1.8]{FKW}, we may assume that $Y'$ is an extension of a constant \'etale group scheme by a
constant multiplicative group scheme, and that $\varepsilon$ is an isomorphism of extensions.
By \cite[3.1.4, 3.1.5]{FKW}, the extension $Y'$ is syntomic, and we may assume that
the isomorphism $f^*Y' \iso g^*\A[p]$ extends to an isomorphism of finite flat group schemes (which automatically respects the extension structure) over $\Spec B.$

Now let $h: B \rightarrow k[\epsilon]$ be any map which vanishes on the image of $\bar \gm_C/\bar \gm_C^2,$ so that $h$ induces
the constant map $C \rightarrow k.$ Then $h^*f^*Y' \iso h^*g^*\A[p]$ is a split extension over $\Spec k[\epsilon].$
It follows from \cite[3.2.2]{FKW} that $h\circ g(\gm_A) = 0,$ which proves the inclusion in the lemma.
\end{proof}

\begin{para} We introduce the following notation. For a map $f:X \rightarrow Y$ of smooth $k$-schemes, we let
$$ r(f) = \max_{x \in X(k)} \dim_k \Im\!\!(\bar \gm_{f(x)}/\bar \gm^2_{f(x)} \rightarrow \bar \gm_x/\bar \gm^2_x) $$
For $f$ a map of smooth $\O_K$-schemes, set $r(f) = r(f\otimes k).$ Note that $r(f)$ does not change if we restrict $f$ to a dense open subset in $X.$
\end{para}

 \begin{prop}\label{prop:essdimAg} Let $\ZZ$ be a smooth, connected $\O_K$-scheme, and let 
 $\ZZ \rightarrow A_{g,N/\O_K}$ be a map of $\O_K$-schemes such that the image of the special fiber, $\ZZ_k$, meets the ordinary locus
 $\A_{g,N}^{\ord} \subset \A_{g,N/k}.$
Then
$$ \ed(\A[p]|_{\ZZ_K}/\ZZ_K; p) \geq r(f) $$
\end{prop}
\begin{proof}  The proof of this is almost the same as that of Theorem \cite[3.2.6]{FKW}. The only difference is that we use
Lemma \ref{lem:basicedestimate} instead of Lemma 3.2.4 of {\it loc.~cit} at the end of the proof.
\end{proof}

\begin{example}
	Let $\H_g$ denote the moduli of hyperelliptic curves of genus $g$.  Let $\H_g[n]$ be the moduli of pairs $(C,\B)$ where $C$ is a hyperelliptic genus $g$ curve and $\B$ is a symplectic basis for $H_1(C;\Z/n\Z)$. Let $\tau\colon \H_{g/\O_K}\to\A_{g/\O_K}$ denote the Torelli map.  By \cite[Theorem 1.2]{Lan}, $\tau$ is an embedding only when the characteristic of $k$ is prime to $2$; when $k$ is of characteristic 2, $r(\tau)=g+1$. Because of this, \cite[Theorem 3.2.6]{FKW} does not give a lower bound on $\ed(\H_g[2]/\H_g;2)$. Using Proposition~\ref{prop:essdimAg} above instead, as well as the argument of \cite[Corollary 3.2.7]{FKW}, we obtain
	\[
	\ed(\H_g[2]/\H_g;2)\ge g+1.
	\]
\end{example}

\begin{rem}
	More generally, Proposition~\ref{prop:essdimAg} gives an arithmetic tool for obtaining lower bounds on the essential dimension at $p$, analogous to the ``fixed point method'' (cf. \cite{ReICM}). As forthcoming work of Brosnan-Fakhrudin-Reichstein \cite{BFR} demonstrates, the fixed point method applied to the toroidal boundary recovers the bounds of Theorem~\ref{it:main} and similar bounds for non-compact locally symmetric varieties (including those not of Hodge type); it also allows one to use toroidal boundary components other than those corresponding to Siegel parabolics. However, as remarked in \cite{FKW}, we are not aware of methods besides Proposition~\ref{prop:essdimAg} that apply to unramified nonabelian covers of compact varieties.
\end{rem}

\begin{para} Proposition \ref{prop:monodromy} implies that the monodromy group of $\A_{g,pN} \rightarrow \A_{g,N}$ can be identified with $\Sp_{2g}(\F_p)$.  Fix such an identification. Let $G$ be a subgroup of $\Sp_{2g}(\F_p)\subset\Sp_{2g}(\Z/pN\Z).$  Denote by $\A_{g,G} \rightarrow \A_{g,N}$ the 
finite, normal, covering corresponding to $G.$
\end{para}

\begin{thm}\label{thm:essdimsbgpG} Let $p$ be a prime, and let $N\geq 3$ be prime to $p$. $G\subset \Sp_{2g}(\F_p)\subset\Sp_{2g}(\Z/pN\Z)$. Then
$$ \ed(\A[p]|_{\A_{g,G}}/\A_{g,G}; p) \geq \max_{U} \dim_{\F_p} U\cap G, $$
where the maximum on the right hand side is over all unipotent radicals of Siegel parabolics in $\Sp_{2g}(\F_p)$.
\end{thm}
\begin{proof} 
	Let $U_0\subset\Sp_{2g}(\F_p)$ be an abelian unipotent subgroup such that $\dim_{\F_p} U_0\cap G$ achieves the maximum. Let $U\subset\Sp_{2g}/_{\F_p}$ be the abelian unipotent subgroup defined in Proposition~\ref{prop:monodromy}. Because all Siegel parabolics are conjugate in $\Sp_{2g}(\F_p)$, there exists a conjugate of $G$, denoted $G'\subset\Sp_{2g}(\F_p)$, such that 
	\[
		\dim_{\F_p} U(\F_p)\cap G'=\dim_{\F_p} U_0\cap G.
	\]
	 Because conjugate subgroups give isomorphic covers, and because $\ed(-;p)$ is a birational invariant, 
	\[
		\ed(\A[p]|_{\A_{g,N,G}}/\A_{g,N,G};p)=\ed(\A[p]|_{\A_{g,N,G'}}/\A_{g,N,G'};p).
	\]
	It therefore suffices to prove the theorem under the assumption that $U_0=U(\F_p)$. For this, it suffices to consider the case $G = U(\F_p)\cap G.$ In the following we slightly abuse notation and write $U$ for $U(\F_p).$
	
Let $x \in \A_{g,N}(k)$ be a point in the ordinary locus. By (2) of  Proposition \ref{prop:monodromy}, there exists $y \in \A_{g,pN}(k)$ and $x' \in A_{g,N,U}(k)$ with $y$ mapping to $x'$ and $x,$ such that the natural map
$$ A : = \widehat \O_{\A_{g,N}, x} \rightarrow \widehat \O_{\A_{g,N,U}, x'}$$
is an isomorphism, and such that, if $B =  \widehat \O_{\A_{g,pN}, y},$ then
$$ \Spec B[1/p] \rightarrow \Spec A[1/p] $$
is a $U$-covering.

Let $\UU = \Hom_{\F_p}(U,\F_p),$ and $\UU'_G = \Hom_{\F_p}(U/(U\cap G), \F_p).$
By Corollary \ref{cor:monodromy}, $\UU$ is identified with a nondegenerate subspace of $\Ext^1_X(\Z/p\Z, \mu_p)$
where $X = \Spec A[1/p].$ Now let $A' = \widehat \O_{\A_{g,N,U\cap G},\bar x''},$ where $x''$ denotes the image of $y$
in $\A_{g,N,U\cap G}.$ Since $\A_{g,N,U\cap G}$ is normal, using the notation of \ref{para:extns}, we have $A' = A(\UU'_G).$
Hence, by Lemma \ref{lem:span}, we have
$$ \dim_k \, \Im\!\!(\bar\gm_A/\bar\gm_A^2 \rightarrow \bar \gm_{A'}/ \gm^2_{A'}) = \dim_{\F_p} U\cap G,$$
and $A'$ is a power series ring over $\O_K.$

Since $x$ was any point in the ordinary locus, this shows that $r(f) \geq \dim_{\F_p} U\cap G,$ where
$f:\A_{g,N,U\cap G} \rightarrow \A_{g,N},$ and that $\A_{g,N,U\cap G}$ is smooth over $\O_K$, over the ordinary locus of
$\A_{g,N}.$ Combining this with Proposition \ref{prop:essdimAg} proves the theorem.
\end{proof}

\section{Modular symplectic representations of finite groups}\label{s:group}

\subsection{General Finite Groups}
Let $p$ be prime, $G$ a finite group and $V$ a faithful, finite-dimensional $G$-representation over $\F_p.$ 
The pairing
$$ \ev\colon V\otimes V^\vee\to\F_p $$
extends to a $G$-invariant symplectic form on $V\oplus V^\vee$. We refer to the associated representation
\[
    G\to \Sp(V\oplus V^\vee)
\]
as the {\em diagonal (symplectic) representation} associated to $V$.

\begin{lemma}\label{p:gensymp} Let $H\subset G$ be an elementary abelian $p$-subgroup, such that $H$ maps to the unipotent radical of a maximal parabolic in $\GL(V).$ Then there exists a Siegel parabolic of $P\subset \Sp(V\oplus V^\vee)$ with unipotent radical $U$ such that, under the diagonal representation associated to $V$,
    \begin{equation*}
        H \subset U\cap G.
    \end{equation*}
\end{lemma}
\begin{proof}  Any maximal parabolic  in $\GL(V)$ is the stabilizer $P(W)$ of a subspace $W\subset V.$  
Let $U(W)$ denote the unipotent radical of $P(W).$
Let $W^\perp\subset V^\vee$ denote the dual subspace. 
Then $W\oplus W^\perp$ is a Lagrangian subspace of $V\oplus V^\vee$, and
    \begin{equation*}
        \GL(V)\cap\Stab_{\Sp(V\oplus V^\vee)}(W\oplus W^\perp)=\Stab_{\GL(V)}(W) = P(W).
    \end{equation*}
    Hence 
     \begin{equation*}
        \GL(V)\cap U(W\oplus W^\perp)=U(W),
    \end{equation*}
    where $U(W\oplus W^\perp)$ is the unipotent radical of 
    $\Stab_{\Sp(V\oplus V^\vee)}(W\oplus W^\perp),$ the Siegel parabolic 
    corresponding to $ W\oplus W^\perp.$
    In particular $H \subset U(W) \subset U(W\oplus W^\perp),$ 
    the Siegel parabolic corresponding to $W\oplus W^\perp.$
     \end{proof}

\begin{para} Let
\[ 
s_p(G): = \max_{U\subset \GL(V)}\dim_{\F_p} U \cap G
\]
where the maximum is taken over all faithful representations $G$ of $V,$ and unipotents $U$ of maximal parabolics in $\GL(V).$ Proposition~\ref{p:gensymp} and Theorem~\ref{thm:essdimsbgpG} immediately imply the following. 
\end{para}

\begin{corollary}\label{c:gen}
    For some $g,$ there exists a congruence cover $\A_{g,p}\to \A_{g,G}$ with
    \[
        \ed(\A_{g,p}/\A_{g,G};p)\ge s_p(G).
    \]
\end{corollary}

%

\begin{remark} 
 While Corollary~\ref{c:gen} implies that $\ed(G;p)\ge s_p(G)$, this is not hard to show directly, e.g. by \cite[Lemma 4.1]{BR}. 
 In fact,  let 
 \[
    r_p(G):=\max_{H\subset G}\dim_{\F_p}H
\]
where the maximum is taken over all elementary abelian $p$-groups $H\subset G.$  
Then $\ed(G;p) \ge  r_p(G) \ge s_p(G).$
The novelty of Corollary~\ref{c:gen} is that a) this lower bound can be realized by an explicit congruence cover;  and b) the congruence cover, and thus the lower bound, comes from modular representation theory at the relevant prime, rather than from ordinary representation theory in characteristic 0 (as in e.g. \cite{BR} or the theorem of Karpenko-Merkurjev \cite{KarMer}). 

The corollary is most interesting in those cases where $s_p(G)$ is large. In the remainder of this section we give examples where 
$s_p(G)$ is equal to, or at least very close to $r_p(G).$ These consist of the case of alternating groups when $p=2,$ and the case where 
$G$ is the $\F_q$-points of a split semisimple group of classical type.
\end{remark}

\subsection{The Groups $S_n$ and $A_n$} \label{s:An}
We now specialize to the symmetric groups $S_n$ and the alternating groups $A_n.$ 
%
%
%
%
%

\begin{para} We would like to apply Corollary \ref{c:gen} to the case of symmetric and alternating groups.
Meyer-Reichstein \cite[Corollary 4.2]{MR} proved that $\ed(S_n;p)=r_p(S_n)$ and similarly for $A_n$ for all $n$ and $p$. 
However, in Appendix~\ref{appendix}, Harman shows that for $p > 2,$ $s_p(S_n)<r_p(S_n)$ and similarly for $A_n$. The purpose of this 
section is to show - see Proposition \ref{p:di} below - that one has $s_2(S_n) = r_2(S_n)$ for all $n$, and $s_2(A_n)=r_2(A_n)$ (resp. $s_2(A_n)=r_2(A_n)-1$) for $n=2,3$  (resp.~$0,1$) modulo $4.$ This uses a remarkable mod $2$ symplectic representation of $S_n,$ discovered by 
Dickson. Harmon's results imply that for $n \geq 5,$ this is the only mod $2$ representations for which the unipotent 
of a maximal parabolic meets $S_n$ in a maximal elementary abelian $2$-group. 

 Recall the ``permutation irrep'' $V$ of $S_n$ over $\F_p$.\footnote{The results of Dickson \cite{Di} and Wagner \cites{Wa1,Wa2} show that the permutation irrep is a minimal-dimensional faithful irrep for $n>8$ and $p=2$, or for $n>6$ and $p$ odd.} For $p\nmid n$ this is the analogue over $\F_p$ of the standard permutation irrep in characteristic 0, i.e. the invariant hyperplane
\[
    V=\{(a_1,\ldots,a_n)\in\F_p^n~|~\sum a_i=0\}
\]
For $p\mid n$ the diagonal line $\Delta:=\{(a,\ldots,a)\}\subset \F_p^n$ is an invariant subspace of the invariant hyperplane, and
\[
    V=\{(a_1,\ldots,a_n)\in\F_p^n~|~\sum a_i=0\}/\Delta.
\]
Dickson \cite{Di} showed that over $\F_2$, the permutation irrep of $S_n$ is a symplectic representation. Let
\[
	d_n:=\lceil \frac{n}{2}\rceil -1,
\]
so that Dickson's representation gives a ``Dickson embedding'' $S_n\subset\Sp_{2d_n}(\F_2)$.

\end{para}

\begin{prop}\label{p:di}\mbox{} Let $N\ge 3$ be odd. For all $n\geq 2$, consider the Dickson embedding $S_n\subset\Sp_{2d_n}(\F_2)\subset \Sp_{2d_n}(\Z/2N\Z)$. There exists a Siegel parabolic with unipotent radical $U$ such that 
        	\begin{align*}
        		\dim_{\F_2}U\cap S_n&=\lfloor \frac{n}{2}\rfloor,\\
        		\dim_{\F_2}U\cap A_n&=\lfloor \frac{n}{2}\rfloor-1.
        	\end{align*}
        By Theorem~\ref{thm:essdimsbgpG}, for all $n\geq 1$: 
            \begin{align*}
            	\ed(\A_{d_n,2N}/\A_{d_n,S_n};2)&=\lfloor \frac{n}{2}\rfloor=\ed(S_n;2),\\
                \ed(\A_{d_n,2N}/\A_{d_n,A_n};2)&=\lfloor \frac{n}{2}\rfloor -1,
            \end{align*}
            i.e.
            \begin{equation*}
            	\ed(\A_{d_n,2N}/\A_{d_n,A_n};2)=\left\{\begin{array}{ll}
            	 	\ed(A_n;2)-1 & n=0,1 \text{ mod }4\\
            	 	\ed(A_n;2) & n=2,3 \text{ mod }4\\
            	\end{array}\right.
            \end{equation*}
            
\end{prop}
\begin{proof}
   Let $V$ denote the permutation irrep of $_n$ over $\F_2$, as in \cite{Di}, i.e. 
    \[
    	V=H/\Delta=\{(x_1,\ldots,x_{2\lceil\frac{n}{2}\rceil})\in \F_2^{2\lceil\frac{n}{2}\rceil}~|~\sum_i x_i=0\}/\{(x,\ldots,x)\in \F_2\}
    \]
    A convenient basis for $V$ is given by the cosets in $H$ of
    \begin{align*}
        e_i&:=[(\underbrace{0,\ldots,1,\ldots,0}_{\text{1 in the $ith$ place}},0,1)]
    \end{align*}
    for $i=1,\ldots,2d_n$.  With respect to this basis the action of $S_{2d_n}\subset S_n$ is the standard permutation action of $S_{2d_n}$ on $\F_2^{2d_n}$. Dickson \cite[p. 124]{Di} proved that the $S_n$ action on $\F_2^{2d_n}$ preserves the symplectic form $\sum_{1\le i\neq j\le d_n} x_iy_j$. We now change basis for ease of studying a Lagrangian. Let
    \begin{equation*}
        \begin{matrix}
            \omega_i:=e_{2i-1}+e_{2i}\\
            \\
            \omega_i^\vee=\sum_{j=0}^{2i-1} e_j.
        \end{matrix}
    \end{equation*}
    A straightforward computation shows that the planes $W=\langle \{\omega_i\}_{i=1}^n\rangle$ and $W^\perp=\langle \{\omega_i^\vee\}_{i=1}^n\rangle$ are dual Lagrangians written with dual Lagrangian bases.

    Now fix $W$ and let $P:=\Stab(W)$ be the corresponding Siegel parabolic with unipotent $U$. From the Lagrangian basis for $W$, we see that
    \begin{equation}
    \label{eq:unip4}
        \F_2^{\lfloor \frac{n}{2}\rfloor}=\langle (12),(34),\ldots,(2\lfloor \frac{n}{2}\rfloor -1~2\lfloor\frac{n}{2}\rfloor )\rangle \subset U\cap S_n
   \end{equation}
    But this is a maximal elementary abelian $2$-group in $S_n$, so \eqref{eq:unip4} is an equality.  Thus 
        \begin{align*}
        U\cap A_n=\langle (12)(34),\ldots, (12)(2\lfloor \frac{n}{2}\rfloor -1~2\lfloor\frac{n}{2}\rfloor)\rangle=\F_2^{\lfloor \frac{n}{2}\rfloor -1}
    \end{align*}
    as claimed.
\end{proof}

\subsection{Finite groups of Lie type}

\begin{prop}\label{prop:Liegpbound} Let $q = p^r,$ and $G = H(\F_q),$ where $H$ is one of the semisimple Lie groups 
$\SL_m, \SO_{2m+1},  \Sp_{2m}$ with $m\geq 2$ or  $\SO_{2m},$ with $m \geq 4.$ 
Let $\rho: H \rightarrow \GL(V)$ be the standard representation of $H$ over $\F_q.$
 Then there exists a parabolic $P(W) \subset \GL(V)$ with unipotent radical $U,$ such that $\dim_{\F_q} W = \lfloor \frac {\dim V} 2  \rfloor,$  
 and $r'_p(G) : = \dim_{\F_q} G \cap U$ satisfies:
  \begin{itemize}
    \item If $G = \SL_m(\F_q),$ then $r'_p(G) = \lfloor \frac {m^2} 4   \rfloor.$ 
    \item If $G = \Sp_{2m}(\F_q)$ then $r'_p(G) = \frac {m(m+1)} 2.$ 
    \item If $G = \SO_{2m}(\F_q)$ then $r'_p(G) = \frac {m(m-1)} 2.$ 
    \item If $G = \SO_{2m+1}(\F_q)$ then  $r'_p(G) = \frac {m(m-1)} 2.$ 
   \end{itemize}
  We have $r\cdot r'_p(G) = r_p(G)$ in all cases except if $G = \SO_{2m+1},$ 
  in which case $r_p(G)/r = \frac {m(m+1)} 2$ if $q$ is even 
   and  $r_p(G)/r = \frac {m(m-1)} 2 + 1$ (resp.~ $5$, resp.~$3$) if $q$ is odd and 
  $m\geq 4,$ (resp.~$m=3$, resp.~$m=2$).
\end{prop}

\begin{proof} We use the standard representations of the root systems of each of the groups $H.$ 
In each case, we will recall the weights appearing in $V,$ specify the subspace $W \subset V,$ and describe a subgroup $U_G \subset H$ as a sum of root spaces. In each case if $r$ is a root appearing in $U_G$ and $w,w'$ are weights appearing in $W$ and $V/W$ respectively, 
then $r+w$ does not appear in $V,$ and $r+w'$ does not appear in $V/W.$ This implies that $U_G \subset H \cap U.$

If $G = \SL_m(\F_q),$ then the weights of $V$ are $e_1, \dots, e_m,$ and $W = \langle e_1, \dots, e_{\lfloor \frac m 2 \rfloor}\rangle.$ 
The roots appearing in $U_G$ are $e_i - e_j$ with $ i \leq \lfloor \frac m 2 \rfloor < j.$

If $G = \Sp_{2m}(\F_q),$ then the weights of $V$ are $\pm e_1, \dots, \pm e_m,$ and $W = \langle e_1, \dots, e_m \rangle.$ 
The roots appearing in $U_G$ are $e_i+e_j$ and $2e_i$ for $1 \leq i < j \leq m.$

If $G = \SO_{2m}(\F_q),$ then the weights of $V$ are $\pm e_1, \dots, \pm e_m,$ and $W = \langle e_1, \dots, e_m \rangle.$ 
The roots appearing in $U_G$ are $e_i+e_j$ for $1 \leq i < j \leq m.$

If $G = \SO_{2m+1}(\F_q),$ then the weights of $V$ are $\pm e_1, \dots, \pm e_m, 0$ and $W = \langle e_1, \dots, e_m \rangle.$ 
The roots appearing in $U_G$ are $e_i+e_j$ for $1 \leq i < j \leq m.$

The maximal elementary abelian $p$-subgroups of $H(\F_q)$ for each group $H$ appearing above are computed in 
\cite{Barry}. In particular, for $G$ equal to one of $\SL_m(F_q), \Sp_{2m}(\F_q), \SO_{2m}(\F_q),$ one sees that $U_G$ 
is already a maximal elementary abelian $p$-subgroup, so that $U_G = H \cap U$ and $r\cdot r'_p(G) = r_p(G).$ 
For $G = \SO_{2m+1}(\F_q)$ the claims about $r_p(G)$ also follows from {\em loc.~cit}, and it remains only to 
prove that $U_G = H \cap U$ in this case. 

To see this, consider $v = \sum_r a_r r \in \Lie (H \cap U)$ where $r$ is a positive root of $H$  and $a_r$ is a scalar. 
Now $V$ is a cyclic highest weight module for $\Lie H.$ Using this and that $v$ annihilates $e_j \in W,$ one gets $a_r = 0$ if 
$r = e_i - e_j.$ Similarly, since $v$ annihilates $-e_j \in V/W,$ $a_r = 0$ for $r = e_j.$ Thus $v \in U_G.$
\end{proof} 

\begin{remark} Note that when $q$ is even, one has $ \SO_{2m+1}(\F_q) \simeq \Sp_{2m}(\F_q),$ 
so that $s_p(G) = r_p(G)$ in this case.
\end{remark}

\section{Classical Problems and Congruence Covers}\label{s:equiv}
Beginning with the work of Hermite on the quintic \cite{Hermite}, the use of modular functions to solve algebraic equations is a major theme of 19th century work, including Klein's icosahedral solution of the quintic \cite{KleinIcos}, the Klein-Burkhardt formula for the 27 lines on a cubic surface \cites{KleinLetter,Burkhardt1,Burkhardt2,Burkhardt3}, the Klein-Gordan solution of equations with Galois group the simple group $\PSL(2,7)$ \cites{Klein8,Gordan}, and the Klein-Fricke solution of the sextic \cites{KleinLast,Fricke}. Underlying this work is the fact that problems of algebraic functions are often {\em equivalent} to problems of modular functions and congruence covers.  

Our goal in this section is to record the classical equivalences, and add to them using recent advances in uniformization. We begin by axiomatizing the notion of accessory irrationality, and recalling the general context in which to take up Klein's call to ``fathom the nature and significance of the necessary accessory irrationalities'' \cite[p. 174]{KleinIcos}. We then recall the general setup of congruence covers of locally symmetric varieties in order to state the precise equivalences. 

While many of the results of this section are implicit in the classical literature, as far as we can tell, with the exception of Klein's {\em Normalformsatz} \cite{KleinIcos}, that various classical problems are in fact {\em equivalent} has gone unremarked in the literature until quite recently \cite{FW}.

\subsection{Accessory Irrationalities and $\E$-Versality} For the rest of the paper we fix an algebraically closed field $K$ of characteristic $0.$ 

\begin{para} By a {\em branched cover} $Y \rightarrow X,$ we mean a dominant, finite map of normal $K$-schemes of finite type. 
Branched covers form a category: a map $(Y' \rightarrow X') \rightarrow (Y \rightarrow X)$ is a commutative diagram
$$\xymatrix{ Y' \ar[r]\ar[d] & Y\ar[d] \\ 
X' \ar[r] & X.
}$$

If $f: X' \rightarrow X$ is a map of normal $K$-schemes of finite type, denote by $f^*Y$ the normalization of $Y\times_XX'.$
If $X$ is connected then $Y \rightarrow X$ corresponds to a finite set $S_Y$ with an action of $\pi_1(U)$ for some dense open $U \subset X,$ where $\pi_1(U)$ denotes the \'etale fundamental group of $U.$ 
We denote by $\Mon(Y/X)$ the image of $\pi_1(U)$ in $\Aut(S_Y).$
\end{para}

\begin{para}
We now introduce the notion of a class of {\em accessory irrationalities} (cf.~Klein \cites{KleinIcos,KleinNU}, see also Chebotarev \cite{ChICM}). 

\begin{defn}[{\bf Accesory irrationalities}] 
\label{definition:accesory} 
A {\em class of accessory irrationalities} is a full subcategory $\E$ of the category of branched covers. 
If $\E(X) \subset \E$ denotes the subcategory consisting of branched covers $\tilde X \rightarrow X,$ then 
we require that $\E(X)$ is stable under isomorphisms, and satisfies the following conditions.
\begin{enumerate}
\item For any $X,$ the identity $X\rightarrow X$ is in $\E(X).$
\item For any map $f:X' \rightarrow X$ of normal $K$-schemes of finite type, $f^*$ induces a functor
$f^*:\E(X) \rightarrow \E(X').$
\item $\E(X\coprod X') = \E(X) \times \E(X').$
\item $\E(X)$ is closed under products: If $E,E' \in \E(X),$ then $E\times_XE' \in \E(X).$
\item If $U \subset X$ is dense open, then the map $\E(X) \rightarrow \E(U)$ induced by restriction is an equivalence of categories. 
\item If $E \rightarrow X' \rightarrow X$ are branched covers and if $E \rightarrow X$ is in $\E(X)$ then $E \rightarrow X'$ is in $\E(X').$
\end{enumerate}
\end{defn}

Axiom (2) implies that $\E$ is a category fibered over the category of normal $K$-schemes. 
Note that Axiom (3) implies that it is enough to specify $\E(X)$ for $X$ connected.

\begin{defn}
Fix a class $\E$ of accessory irrationalities.  The {\em essential dimension} of a cover $\tilde X \rightarrow X,$ with respect to $\E$ is: 
\begin{align*}
	\ed(\tilde{X}/X;\E):=\min_{(E\to X)\in \E}\ed(E\times_X \tilde{X}/E).
\end{align*}
\end{defn}
\end{para}

\begin{example}\label{ex:AI}
	Some of the core classical examples of $\E$ are as follows (for simplicity we specify $\E(X)$ only for $X$ connected):
\begin{enumerate}
	\item For $\E(X)=\{\id:X\to X\}$, the quantity $\ed(\tilde{X}/X;\E)$ is just the essential dimension $\ed(\tilde{X}/X)$. 
	\item Let $p$ be a prime and let $\E(X)$ be the subcategory of branched covers of $X$ whose degree is coprime to $p.$  	
	Then $\ed(\tilde{X}/X;\E)$  is the {\em essential dimension at $p$}. We emphasize that, although it leads to the same notion of essential dimension at $p,$ we do not insist that $E$ is connected, as this version of the definition does not satisfy Axiom (3) of Definition~\ref{definition:accesory}.
	\item Let $\E(X)$ be the set of covers $E\rightarrow X$ with $\Mon(E/X)$ abelian. Then $\ed(\tilde{X}/X;\E)$ is the {\em abelian resolvent degree}. Likewise, we can consider the class of accessory irrationalities with nilpotent (resp. solvable) monodromy, to obtain the {\em nilpotent} (resp. {\em solvable}) resolvent degree (see \cites{KleinNU,ChICM,Ch43}).
	\item Let $G$ be a finite simple group, and let $\E(X)$ consist of all $E\to X$ such that for each connected component $E'$ of $E,$  the branched cover $E' \rightarrow X$ is Galois and a composition series for $\Gal(E'/X)$ has no factor isomorphic to $G.$ We write 
	$\ed(\tilde X/X; G)$ for $\ed(\tilde X/X; \E).$
      \end{enumerate}
\end{example}
    
\begin{defn}[{\bf $\E$-versality}]
\label{d:evers}
		Let $\E$ be a class of accessory irrationalities. A Galois branched cover $\tilde{X}\to X$ with group $G$ is {\em $\E$-versal} if for any other Galois $G$-cover $\tilde{Y}\to Y$, and any Zariski open $U\subset X$, there exists 
		\begin{enumerate}
			\item an accessory irrationality $E\to Y$ in $\E(Y)$, 
			\item a nontrivial rational map $f\colon E\to U$, and 
			\item an isomorphism $f^*\tilde{X}|_U\cong \tilde{Y}|_E$.
		\end{enumerate}
\end{defn}

\begin{remark} If $\E$ is the trivial class of accessory irrationalities, i.e. $\E(X)$ only contains the identity, then $\E$-versal is just ``versal'' in the usual sense of the term (see e.g. \cite[Section 1.5]{GMS}).

If $\E' \subset \E$ are classes of accessory irrationalities, then $\E'$-versality for a $G$-cover implies $\E$-versality. 
In particular a cover which is versal is $\E$-versal for any class $\E.$
\end{remark}

\begin{example}\mbox{}\label{example:H90}
		\begin{enumerate}
			\item Hilbert's Theorem 90 implies that for a finite group $G,$ and every faithful linear action $G\circlearrowleft \AA^n$, the map $\AA^n\to \AA^n/G$ is versal (see \cite{DR1}).
			\item The Merkujev-Suslin Theorem \cite[Theorem 16.1]{MS} implies that for every faithful, projective-linear action $G\circlearrowleft\Pb^n$, the map $\Pb^n\to\Pb^n/G$ is solvably versal, i.e. $\E$-versal for the class $\E$ of solvable branched covers.\footnote{{\em Mutatis mutandis}, this follows by the same reasoning as in \cite{DR1}.}
		\end{enumerate}
\end{example}

\begin{lemma}\label{l:evers}
	Let $G$ be a finite group, let $\E$ be a class of accessory irrationalities, and let $\tilde{X}\to X$ be an $\E$-versal $G$-cover. 
	\begin{enumerate}
		\item Let $\tilde{X}\to \tilde{Z}$ be a $G$-equivariant dominant rational map. Then $\tilde{Z}\to \tilde{Z}/G$ is an $\E$-versal $G$-cover.
		\item Let $H\subset G$ be any subgroup. Then $\tilde{X}\to \tilde{X}/H$ is an $\E$-versal $H$-cover. 
	\end{enumerate}
\end{lemma}
\begin{proof}
		The first statement follows immediately from the definition. For the second, let $\tilde{Y}\to Y$ be a Galois $H$-cover. Then 
		\begin{equation*}
			\tilde{Y}\times_H G\to Y
		\end{equation*}
		is a Galois $G$-cover which is $H$-equivariantly isomorphic to $\tilde{Y}\times G/H\to Y$. By $\E$-versality, for any Zariski open $U\subset X$, there exists an accessory irrationality
		\[
			E\to Y
		\]
		in $\E$, and a rational map
		\[
			f\colon E\to U
		\]
		with an isomorphism of $G$-covers 
		\[
			f^\ast \tilde{X}\cong (\tilde{Y}\times_H G)|_E.
		\]
		By the Galois correspondence for covers, the $H$-equviariant isomorphism above implies that $E\to U$ factors through a map 
		\[
			\tilde{f}\colon E \to (\tilde{X}/H)|_U
		\]
		We conclude that $\tilde{f}^\ast\tilde{X}\cong \tilde{Y}|_E$ and that $\tilde{X}\to\tilde{X}/H$ is $\E$-versal for $H$ as claimed.			
              \end{proof}
          
          \begin{remark}
          \label{remark:M0n:versal}
          Example~\ref{example:H90}(1) and Lemma~\ref{l:evers}(1) immediately imply that for 
          each $ n\geq 4$, the cover $\M_{0,n}\to \M_{0,n}/S_n$ is versal for the group $S_n$.
          \end{remark}

              \begin{para} We can also consider the {\em resolvent degree} of a cover $\tilde{X}\to X$, which is somewhat different from, but related to the idea of the general notion of essential dimension defined above. To explain this, write 
              $E_{\bullet} \rightarrow X$ for a tower of branched covers $E=E_r\to\cdots\to E_0=X.$ 
              The {\em resolvent degree} of $\tilde{X}\to X$ is defined as 
              
		\[
		\RD(\tilde{X}/X)=\min_{E_{\bullet} \rightarrow X}\max\left\{\ed(E\times_X \tilde{X}/E), \left\{\ed(E_i/E_{i-1})\right\}_{i=1}^r\right\}
              \]
where $E_{\bullet} \rightarrow X$ runs over all sequences of covers.

When $\Mon(\tilde{X}/X)$ is simple, it follows from \cite[Cor. 2.18]{FW} that the definition of $\RD(\tilde{X}/X)$ does not change if we consider only $E_{\bullet} \rightarrow X$ such that the composition 
$\pi_1(E)\to \pi_1(X)\to\Mon(\tilde{X}/X)$ is surjective and $\ed(E_i/E_{i-1})<\dim(X)$. 
In particular
\begin{equation}\label{eqn:comparerd}
  \min_{E_{\bullet}\rightarrow X} \ed(E\times_X \tilde{X}/E)\le \RD(\tilde X/X)
  \end{equation}
  where $E_{\bullet}\rightarrow X$ runs over sequences of covers satisfying these conditions.
  On the other hand, in every known example, the current best upper bound for $\RD(-)$ can be exhibited using such a sequence $E_{\bullet}\rightarrow X$ which in addition satisfies $\ed(E\times_X \tilde{X}/E) \geq \ed(E_{i+1}/E_i),$ for $i=1,\dots, r.$

Hilbert  \cites{Hi1,Hi2} made three conjectures on the resolvent degree of the general degree $n$ polynomial; equivalently on 

\[\RD(n):=\RD(\M_{0,n}/(\M_{0,n}/S_n)) = \RD(\M_{0,n}/(\M_{0,n}/A_n)).\]

\begin{iconj}[{\bf Hilbert}]
\label{conjecture:hilbert}
The following equalities hold:

\[
\begin{array}{rl}
\mbox{\it Sextic Conjecture: }& \RD(6)=2.\\
\mbox{\it 13th Problem: }&\RD(7)=3.\\
\mbox{Octic Conjecture: }&\RD(8)=\RD(9)=4.
\end{array}
\]
\end{iconj}
 
The upper bounds in Conjecture~\ref{conjecture:hilbert} are known; the first two are due to Hamilton, the last to Hilbert.

Our interest in $\E$-versality comes from the following lemma, which is proven {\em mutatis mutandis} by the same argument as in the proof of  \cite[Proposition 3.7]{FW}.
\end{para}

\begin{lemma}\label{l:versmax}
	Let $\E$ be a class of accessory irrationalities and let $\tilde{X}\to X$ be an $\E$-versal $G$-cover.  For any Galois branched cover  $\tilde{Y}\to Y$ with monodromy $G$, 

	\[
		\ed(\tilde{Y}/Y;\E)\le \ed(\tilde{X}/X;\E).
	\]
	
	In particular, for any other $\E$-versal $G$-cover $\tilde{X}'\to X'$, 
	
	\[
		\ed(\tilde{X}'/X';\E)=\ed(\tilde{X}/X;\E).
	\]
	
	Further, if $\E$ is any of the classes of Example~\ref{ex:AI} and if $G$ is simple, then 
	
	\[
		\RD(\tilde{Y}/Y)\le\RD(\tilde{X}/X) \ \ \mbox{and}\ \ \RD(\tilde{X}'/X')=\RD(\tilde{X}/X).
	\]
	\end{lemma}

Lemma~\ref{l:versmax} makes precise the classical discovery that $\E$-versal $G$-covers provide ``normal forms''  to which every other $G$-cover or can be reduced.  Notably, for many groups $G$ of classical interest, congruence covers are $\E$-versal for a natural choice of $\E$. 

\begin{rem}
	While the notion of versality has been studied intensively for several decades, many of the most interesting normal forms, beginning with Klein's {\em Normalformsatz}, rely on the notion of {\em solvable versality}, which is substantially more flexible. For example, a versal $G$-variety of minimal dimension must be unirational. On the other hand, there are no rational $A_6$ curves (by Klein's classification of finite M\"{o}bius groups), and the level 3 Hilbert modular surface of discriminant 5, which is solvably versal for $A_6$ and conjectured by Hilbert to be of minimal dimension among such varieties, has arithmetic genus equal to 5 (see the discussion in the proof of Propositon~\ref{p:sext} below). A better understanding of the geometric implications of solvable versality (and related notions) could shed significant light on the underpinnings of Hilbert's conjectures.
\end{rem}

\subsection{$\E$-Versal Congruence Covers}\label{subsection:versal} 

We can now record the $\E$-versal congruence covers that we know. Klein's {\em Normalformsatz} provides the paradigmatic example for what follows.
\begin{para}
Let $G$ be a group-scheme of finite type over $\Z$ whose generic fiber, which we also denote by $G,$ is a connected semisimple group. 
A subgroup $\Gamma \subset G(\Z)$ is called a {\em congruence subgroup} if it contains 
\[G(\Z,n) : = {\rm ker}(G(\Z)\to G(\Z/n))\] 
for some positive integer $n.$

We assume that the quotient $X$ of $G(\RR)$ by its maximal compact subgroup is a Hermitian symmetric domain.
Then for any congruence subgroup $\Gamma,$ a theorem of Baily-Borel asserts that $M_\Gamma:=X/\Gamma$ is a complex, quasiprojective variety. For $\Gamma' \subset \Gamma$ congruence subgroups there is a natural map covering map $M_{\Gamma'} \rightarrow M_{\Gamma}.$ 

For $L$ a totally real number field, one can apply the above to $\Res_{L/\Q} G,$ instead of $G.$ In this case we have 
$G(\O_L) \iso \Res_{L/\Q}G(\Z)$ and when we write $G(\O_L)$ we mean that we are working the Hermitian symmetric domain and congruence subgroups associated with the group $\Res_{L/\Q} G.$
Similarly, we write $G(\O_L,n)$ for $\Res_{L/\Q}G(\Z,n).$ If $L= \Q(\sqrt d)$ is real quadratic, denote by  
$G(\O_L,\sqrt d)$ the kernel of 
$$ \Res_{L/\Q}G(\Z) = G(\O_L) \rightarrow G(\O_L/\sqrt d).$$

If $L$ is quadratic imaginary and $a,b$ are non-negative integers, one can consider the unitary group $\U(a,b)$ of signature $a,b$ 
defined by $L.$ This is the subgroup scheme of $\Res_{\O_L/\Z} \GL_n,$ where $n = a+b,$ which fixes the standard Hermitian 
(with respect to conjugation on $K$) form of signature $(a,b)$. One also has the corresponding projective unitary group $\PU(a,b).$ 

In fact for the rest of this section we take $L = \Q(\omega),$ where $\omega$ is a primitive cube root of $1,$ and we will 
only need groups of signature $n-1,1.$ We denote by $\PU(n-1,1)(\Z,\sqrt{-3})$ the kernel of the composite 
$$ \PU(n-1,1)(\Z) \rightarrow \Res_{\O_L/\Z} \PGL_n(\Z) = \PGL_n(\O_L) \rightarrow \PGL_n(\F_3).$$
\end{para}

\begin{thm}[{\bf Klein's Normalformsatz}, \cite{KleinIcos}]
	Let $\E$ be any class of accessory irrationalities containing all quadratic branched covers. Then the 
	level 5 cover of the modular curve
	\[
		M_{\SL(\Z,5)}\to M_{\SL_2(\Z)}.
	\]
	is an $\E$-versal $A_5$-cover. In particular, for any branched cover $\tilde{X}\to X$ with monodromy $A_5$, 
	\[
		\ed(\tilde{X}/X;\E)=\RD(\tilde{X}/X)=1.
	\]
\end{thm}

This is in contrast to Klein's theorem that $\ed(A_5)=2$.  We can add another example for $A_5$, which was studied in detail by Hirzebruch \cite{Hirzebruch5}, and was likely known to Kronecker, Klein and Hilbert.
\begin{prop}\label{p:Hz}
	The level 2 cover of the Hilbert modular surface 
	\[
		M_{\SL_2(\Z[\frac{1+\sqrt{5}}{2}],2)}\to M_{\SL_2(\Z[\frac{1+\sqrt{5}}{2}])}
	\]
	is versal for $A_5$.
\end{prop}
\begin{proof}
	By \cite{Hirzebruch5}, there is an $A_5$-equivariant birational equivalence
	\[
		\M_{0,5}\simeq M_{\SL_2(\Z[\frac{1+\sqrt{5}}{2}],2)}.
	\]
	There is an $A_5$-equvariant dominant map
	\[
		\AA^5\to \M_{0,5},
	\]
	where the source is the permutation representation and the map is the quotient by the diagonal action of $\Aut(\AA^1)$. By Hilbert's Theorem 90, $\AA^5$ is versal. The proposition now follows from Lemma~\ref{l:evers}. 
\end{proof}

Similar to, but less well-known than, Klein's normalformsatz  is the following (see \cites{Klein8,Gordan}, \cite[p. 318-319]{Hirzebruch77} and \cite[Vol. II, Part 2, Chapters 1-2]{Fricke}).  Denote by $\PSL(2,7)$ 
the image of $\SL_2(\F_7) \rightarrow \PGL_2(\F_7)$; this is a simple group.

\begin{prop}[{\bf Normal forms for $\PSL(2,7)$}]\mbox{}
		\begin{enumerate}
			\item Let $\PGL_2^+(\Z[\sqrt{7}]) \subset \PGL_2(\Z[\sqrt{7}])$ denote the subgroup of elements which 
			lift to an element of $\GL_2(\Z[\sqrt{7}])$ with totally positive determinant.  
						The cover 
				\[
					M_{\PGL_2(\Z[\sqrt{7}],\sqrt{7})}\to M_{\PGL_2^+(\Z[\sqrt{7}])}
				\]
				of Hilbert modular surfaces is versal for the simple group
								 $\PSL(2,7)$.
								 
			\item Let $\E$ be any class of accessory irrationalities containing all $S_4$-covers. Let 	$\widetilde{\SL_2(\Z,7)}$ denote the kernel of the surjection $\SL_2(\Z)\to \PSL(2,7)$. Then the level 7 modular curve
			\[
				M_{\widetilde{\SL_2(\Z,7)}}\to M_{\SL_2(\Z)}
			\]
			is an $\E$-versal $\PSL(2,7)$-cover. In particular, for any branched cover $\tilde{X}\to X$ with monodromy $\PSL(2,7),$ 
			\[
			\ed(\tilde{X}/X;\E)=\RD(\tilde{X}/X)=1.
			\]
	\end{enumerate}
\end{prop}
\begin{proof} We remark that $\Z[\sqrt 7]^\times = \{\pm \epsilon^n \}$ where $\epsilon = 8+3\sqrt 7$ is the fundamental unit. 
Hence $\PGL_2(\Z[\sqrt{7}],\sqrt{7}) \subset \PGL_2^+(\Z[\sqrt{7}]),$ and in particular, the latter group is a congruence subgroup.

	Consider the modular curve $M_{\widetilde{\SL_2(\Z,7)}}$.  This has genus 3, and so the action of $\PSL(2,7)$ on 1-forms gives a linear action of $\PSL(2,7)$ on $\AA^3$, and an equivariant dominant rational map $\AA^3\to \Pb^2$. Lemma~\ref{l:evers} then implies that $\Pb^2$ is versal for $\PSL(2,7).$ As noted on \cite[p. 318-319]{Hirzebruch77}, there is a 
	$\PSL(2,7)$-equivariant birational isomorphism 
		\[
			\Pb^2\cong M_{\PGL_2(\Z[\sqrt{7}],\sqrt{7})}.
		\]
		This proves the first statement of the proposition. 
		
		The second statement follows from \cites{Klein8, Gordan}.  In modern language, it suffices to construct an accessory irrationality $E\to \Pb^2/\PSL(2,7)$ and a $\PSL(2,7)$-equivariant dominant rational map 
		\[
			\Pb^2|_E\to M_{\widetilde{\SL_2(\Z,7)}}.
		\]
		For this, the canonical embedding of the modular curve $M_{\widetilde{\SL_2(\Z,7)}}$ gives a $\PSL(2,7)$-equivariant map
		\[
			M_{\widetilde{\SL_2(\Z,7)}}\to \Pb^2.
		\]
		As Klein discovered, the image of this map is a quartic curve, the so-called ``Klein quartic''. Fixing any $\PSL(2,7)$-invariant pairing on $\Pb^2$, there is a rational map
		\[
			\Pb^2\to \M_{0,4}/S_4.
		\]
		which sends a point $x\in\Pb^2$ to the intersection of the dual line $L_x$ with the Klein quartic. Let 
		\[
			E=\M_{0,4}|_{\Pb^2}.
		\]
		Then there is a $\PSL(2,7)$-equivariant dominant map 
		\[
			\Pb^2|_E\to M_{\widetilde{\SL_2(\Z,7)}}
		\]
		as claimed.
\end{proof}

We can add the following result to the above. 

\begin{prop}[{\bf Normal forms for the sextic}]\label{p:sext}\mbox{}
	\begin{enumerate}
		\item The congruence cover
			\begin{align*}
				\A_{2,2}&\to \A_2\intertext{and the Picard modular 3-fold}
				\M_{\PU(3,1)(\Z,\sqrt{-3})}&\to M_{\PU(3,1)(\Z)}
			\end{align*}
			are versal for $A_6$.\footnote{Recall that there are exceptional isomorphisms $\Sp_4(\F_2)\cong O_4^+(\F_3)\cong  S_6$.}
		\item For $\E$ any class of accessory irrationalities containing all quadratic covers and composites thereof, the congruence cover
			\begin{align*}
				\A_{2,3}/\F_3^\times&\to \A_{2,A_6}\intertext{and the Picard modular 3-fold}
				M_{\PU(3,1)(\Z,2)}&\to M_{\PU(3,1)(\Z,A_6)}
			\end{align*}
			are $\E$-versal for $A_6$.
		\item For $\E$ any class of accessory irrationalities containing all quadratic and cubic covers and composites thereof, the Hilbert modular surface
			\[
				M_{\widetilde{\SL_2(\Z[\frac{1+\sqrt{5}}{2}],3)}}\to M_{\SL_2(\Z[\frac{1+\sqrt{5}}{2}])}
			\]
			is $\E$-versal for $A_6$, where $\widetilde{\SL_2(\Z[\frac{1+\sqrt{5}}{2}],3)}$ denotes the kernel of the map $\SL_2(\Z[\frac{1+\sqrt{5}}{2}])\to\PGL_2(\F_9)=A_6$.
	\end{enumerate}
	
	In particular (cf.\ Remark~\ref{remark:M0n:versal} and Lemma~\ref{l:versmax}) Hilbert's Sextic Conjecture is equivalent to the statement that the resolvent degree of any (and thus each) of the above covers is $\dim(M_{\SL_2(\Z[\frac{1+\sqrt{5}}{2}])})=2$.
\end{prop} 

 Let $\PSp(4,3)$ denote the image of $\Sp_4(\F_3) \rightarrow \PSp_4(\F_3)$; this is a simple group. 
 To prove the proposition, we make use of the following lemma. 
\begin{lemma}\label{l:psp43}
	 Let $\PSp(4,3)$ act linearly on $\Pb^3$ and let $G\subset\PSp(4,3)$ be any subgroup. Let $\E$ be any class of accessory irrationalities containing all composites of quadratic covers. Then $\Pb^3$ is an $\E$-versal $G$-variety.
\end{lemma}
\begin{proof}
	There is (see e.g. \cite{Atlas}) an $\Sp_4(\F_3)$-equivariant dominant rational map
	\[
		\AA^4\to\Pb^3.
	\]
	Lemma~\ref{l:evers} then implies that $\Pb^3$ is versal for $\Sp_4(\F_3)$  As observed in the proof of \cite[Theorem 4.3]{FW}, this implies that $\Pb^3$ is $\E$-versal for $\PSp(4,3)$ and thus, by Lemma~\ref{l:evers} for any $G\subset\PSp(4,3)$ as well.
\end{proof}

\begin{proof}[Proof of Proposition~\ref{p:sext}]
		For versality, as in the proof of Proposition~\ref{p:Hz}, it suffices to prove that there are $A_6$-equivariant birational isomorphisms
		\begin{equation}
		\label{eq:birat6}
			\M_{0,6}\cong \A_{2,2}\cong M_{\PU(3,1)(\Z,\sqrt{-3})}
\end{equation}
		where $\M_{0,6}$ is the moduli of $6$ distinct points in $\Pb^1$. The first isomorphism of \eqref{eq:birat6} is the classical period map which sends 6 points in $\Pb^1$ to the Jacobian of the hyperelliptic curve branched at those points. For the second, consider the {\em Segre cubic threefold} $X_3$ in $\Pb^5$ given by
		\[X_3:=\{[x_0:\cdots :x_5]\in\Pb^5: \sum_{i=0}^5x_i=0=\sum_{i=0}^5x_i^3\}.\]
		The permutation action of $S_6$ on $\Pb^5$ leaves invariant $X_3$, permuting its $10$ nodes. Kondo \cite{Ko} proved that $X_3$ is isomorphic to the Satake-Bailey-Borel compactification of the Picard modular 3-fold $M_{\PU(3,1)(\Z,\sqrt{-3})}$.   One can check that the birational map $M_{\PU(3,1)(\Z,\sqrt{-3})}\to X_3$ is $S_6$-equivariant (cf. e.g. \cite[p. 6, Lemma 2.1]{SBT}).
		
		Hunt proves in \cite[Theorem 3.3.11]{Hu}  that the dual variety to $X_3$ is the so-called {\em Igusa quartic} ${\mathcal I}_4$, which is the moduli space of $6$ points on a conic in $\Pb^2$. The two varieties $X_3$ and ${\mathcal I}_4$ are $S_6$-equivariantly birational. The Igusa quartic ${\mathcal I}_4$ is the Satake compactification of $\A_{2,2}$. The second birational isomorphism in \eqref{eq:birat6} is the composition of these.
		
		Now let $\E$ be any class of accessory irrationalities containing all quadratic covers and composites thereof. As explained in Hunt \cite[Chapter 5.3]{Hu}, there is a $6\colon1$ (in particular, dominant) $\PSp(4,3)$-equivariant rational map
		\[
			\Pb^3\to \B
		\]
		where the action of $\PSp(4,3)$ on $\Pb^3$ is linear and where $\B$ denotes the ``Burkhardt quartic''. There is also a $\PSp(4,3)$-equivariant birational isomorphism
		\[
			\B\cong \A_{2,3}/\F_3^\times. 
		\]		
		Lemma~\ref{l:psp43} implies that $\A_{2,3}/\F_3^\times$ is $\E$-versal for $G=A_6\subset \PSp(4,3)$.
		
		Thus $\A_{2,3}/\F_3^\times$ is $\E$-versal for any subgroup of $\PSp(4,3)$, in particular $A_6$. Finally, Hunt \cite[Theorem 5.6.1]{Hu} proved that $\B$ is $\PSp(4,3)$-equivariantly biregularly isomorphic  to the Baily-Borel compactification of $M_{\PU(3,1)(\Z,2)}$.
		
		For the last statement, let $\E$ be any class of accessory irrationalities containing all composites of quadratic and  cubic covers. By \cite[Chapter VIII, Theorem 2.6]{vdG}, there exists an $A_6$-equivariant birational isomorphism
		\[
			M_{\SL_2(\Z[\frac{1+\sqrt{5}}{2}],3)}\cong V_{1,2,4}\subset\Pb^5
		\]
		where $V_{1,2,4}$ is the common vanishing locus of the 1st, 2nd and 4th elementary symmetric polynomials and $A_6$ acts on $\Pb^5$ via the permutation representation. As above, it suffices to construct an accessory irrationality $E\to\AA^6/A_6$ in $\E$ and an equivariant dominant rational map 
		\[
			\AA^6|_E\to V_{1,2,4}.
		\]
		This follows from the classical theory of Tschirnhaus transformations (see e.g. \cite{W} for a contemporary treatment). Recall that a {\em Tschirnhaus transformation} $T_{\bf b}$, for some ${\bf b}:=(b_0,\ldots,b_5)\in\AA^6$, is the assignment which sends a root $z$ of the generic sextic to
		\[
			\sum_{i=0}^5 b_iz^i.
		\]
		This defines an $S_6$-equivariant rational map
		\[
			T_{\bf b}\colon \AA^6\to\AA^6
		\]
		Letting ${\bf b}$ vary, we have an $\AA^6_{\bf b}$ parameter space of Tschirnhaus tranformations for each sextic, which we view as a trivial bundle
		\[
			\pi_1\colon \AA^6\times\AA^6_{\bf b}\to \AA^6.
		\]
		We also have an evaluation map 
		\[
			\ev\colon \AA^6\times\AA^6_{\bf b}\to\AA^6
		\]
		By direct computation (see e.g. \cite[Definition 3.5 and Lemma 3.6]{W}), $\ev^{-1}(\widetilde{V_{1,2,4}})$, i.e. the preimage under the map $\ev$ of the affine cone over $V_{1,2,4}$,  is the intersection of a (trivial) family of hyperplanes $\widetilde{T_1}$, a cone over a generically smooth quadric $\widetilde{T_2}$ (for smoothness, see e.g. \cite[Lemma 2.6]{W}), and a quartic cone $\widetilde{T_4}$. By the classical theory of quadrics (e.g. \cite[Lemma 5.10]{W}), there exists a finite, generically \'etale map
		\[
			E_0\to \AA^6/A_6
		\]
		with monodromy a 2-group such that the quadric cone $\widetilde{T_2}|_{E_0}$ contains a (trivial) family $\L\to E_0$ of 2-planes over $E_0$. The intersection
		\[
			\L\times_{E_0}\widetilde{T_4} 
		\]		
		is thus the affine cone over a family of 4 distinct points in $\Pb^1$. There thus exists an $S_4$-cover
		\[
			E\to E_0	
		\]
		and a section $\sigma\colon E\to \ev^{-1}(\widetilde{V_{1,2,4}})|_E$. The map
		\[
			\ev\circ \sigma\colon \AA^6|_E\to V_{1,2,4}
		\]
		gives the dominant map we seek. By construction $E\to \AA^6/A_6$ is in the class $\E$, and thus $V_{1,2,4}$ is indeed $\E$-versal. 
\end{proof}

\begin{prop}[{\bf Normal forms for the 27 lines}]\mbox{}
	\begin{enumerate}
		\item The congruence cover
			\begin{align*}
				\M_{\PU(4,1)(\Z,\sqrt{-3})}&\to M_{\PU(4,1)(\Z)}
			\end{align*}
			is versal for $O_5^+(\F_3)\cong W(E_6).$ \footnote{Recall  that there is an exceptional isomorphism of $O_5^+(\F_3)$ with the Weyl group of $E_6$.}
			
		\item For $\E$ any class of accessory irrationalities containing all quadratic covers and 	composites thereof, the congruence cover
		\begin{align*}
			\A_{2,3}/\F_3^\times&\to \A_2\intertext{and the Picard modular 3-fold}
			M_{\PU(3,1)(\Z,2)}&\to M_{\PU(3,1)(\Z)}
		\end{align*}
		are $\E$-versal for the simple group $\PSp(4,3)\cong W(E_6)^+$ (the normal index 2-subgroup of $W(E_6)$).
	\end{enumerate}
	
	In particular, \cite[Conjecture 1.8]{FW} implies and is implied by the resolvent degree of any (and thus each) of the above covers equaling $\dim(\A_2)=3$. 
\end{prop} 
\begin{proof}
		By Allcock-Carlson-Toledo \cite{ACT}, there exists an $O_5^+(\F_3)\cong W(E_6)$-equivariant birational isomorphism
		\[
			\H_{3,3}(27)\simeq M_{\PU(4,1)(\Z,\sqrt{-3})}
		\]
		from the moduli $\H_{3,3}(27)$ of smooth cubic surfaces with a full marking of the intersection of their 27 lines to the Picard modular 4-fold. 
		
		By \cite[Lemma 6.1]{DR2} $\H_{3,3}(27)$ is versal for $W(E_6)$. By Lemma~\ref{l:evers}, both varieties are therefore versal for any subgroup of $W(E_6)$.  
		
		The remaining statements follow from the proof of Proposition~\ref{p:sext} above. Concretely, there we showed that $\A_{2,3}/\F_3^\times$ was $\E$-versal for any subgroup of $\PSp(4,3)$, in particular for $\PSp(4,3)$ itself. Together with the $\PSp(4,3)$-equivariant birational isomorphism
		\[
			\A_{2,3}/\F_3^\times\simeq M_{\PU(3,1)(\Z,2)}
		\]
		recalled in the proof of Proposition~\ref{p:sext}, this implies the result.
\end{proof}

\begin{prop}[{\bf Normal forms for the septic, the octic, and 28 bitangents}]\label{p:8}\mbox{}
	Let $G\subset \Sp_6(\F_2)$ be any subgroup. Then the cover
	\[
		\A_{3,2}\to \A_{3,G}
	\]
	is versal for $G$. In particular (cf.\ Remark~\ref{remark:M0n:versal} and Lemma~\ref{l:versmax}) : 
	
	\begin{enumerate}
		\item Hilbert's 13th Problem is equivalent to 
		\[
			\RD(\A_{3,2}/\A_{3,A_7})=3.
		\]
		\item Hilbert's Octic Conjecture \cite[p. 248]{Hi2} is equivalent to 
			\[
				\RD(\A_{3,2}/\A_{3,A_8})=4.
			\]
		\item \cite[Problem 5.5(2)]{FW}, which asks for the resolvent degree of finding a bitangent on a planar quartic, is equivalent to asking for $\RD(\A_{3,2}\to \A_3)$.
	\end{enumerate}
\end{prop}
\begin{proof}
	This follows as in the proof of \cite[Proposition 5.7]{FW}, which in turn draws on \cite{DO} and ideas of Coble. As explained there, there exists an $\Sp_6(\F_2)$-equivariant dominant rational map
	\[
		\AA^7\to\H_{4,2}(28)\simeq  \mathcal{M}_3[2]
	\]
	where $\H_{4,2}(28)$ denotes the moduli of smooth planar quartics with a marking of their 28 bitangents, and $\mathcal{M}_3[2]$ denotes the moduli of genus 3 curves with full level 2 structure. The period mapping gives an $Sp_6(\F_2)$-equivariantly birational isomorphism 
	\[
		\mathcal{M}_3[2]\simeq \A_{3,2},
	\]
	and thus $\A_{3,2}$ is a versal $G$-variety for any $G\subset\Sp_6(\F_2)$ as claimed.
\end{proof}

\begin{para} By Klein \cite{Klein87}, the action $A_7\circlearrowleft \Pb^3$ is solvably versal. As a result, Hilbert's 13th problem is equivalent to the assertion that the cover $\Pb^3\to\Pb^3/A_7$ is a normal form of minimal dimension.
\begin{question}
	Is there a congruence cover $X_{\Gamma'}\to X_\Gamma$ with Galois group $A_7$ and $\dim~ X_\Gamma=3$ which is also $\E$-versal for one of the classes of accessory irrationalities considered in Example~\ref{ex:AI}? 
\end{question}
Finding such a congruence cover would give the transcendental part of Klein's 3-variable solution of the degree 7, as in \cite[Chapter 5.9]{KleinIcos}. Note that Prokhorov's classification \cite[Theorem 1.5]{Pr} of finite simple groups acting birationally on rationally connected 3-folds gives strong constraints on any possible congruence cover.

\begin{question}
	Is there a congruence cover $X_{\Gamma'}\to X_\Gamma$ with Galois group $A_8$ and $\dim~ X_\Gamma=4$ which is also $\E$-versal for one of the classes of accessory irrationalities considered in Example~\ref{ex:AI}? 
\end{question}
\end{para}

\begin{para} As Propositions~\ref{p:sext} and \ref{p:8} show, for $g=2,3$ the $\Sp_{2g}(\F_2)$-variety $\A_{g,2}$ is $G$-versal for any subgroup $G \subset \Sp_{2g}(\F_2).$  
	Hence for $n=6,7,8$ the resolvent degree of the cover 
	$\A_{d_n,2} \rightarrow \A_{d_n,A_n}$ is equal to $\RD(n),$ as defined in the introduction. Interestingly, Hilbert's conjectured value for resolvent degree, and the value of the essential dimension at $2$ for these covers, almost agree : 
	\begin{equation*}
	\begin{array}{lcccc}
	n & 6 & 7 & 8 & 9 \\ \hline
	\text{Hilbert:~} \RD(n) & 2 & 3 & 4 & 4\\ \hline
	 \ed(A_n;2) & 2 & 2 & 4 & 4
	\end{array}
	\end{equation*}
	Note that in these cases the value of $\ed(\A_{d_n,2} \rightarrow \A_{d_n,A_n};2) = \ed(A_n;2)$ is already given by
	Proposition~\ref{p:di}, except when $n=8$ and $g = 3,$ in which case Proposition~\ref{p:di} gives the lower bound $3.$ The actual value
	$\ed(\A_{3,2}\to\A_{3,A_8};2)=4$ follows from versality (e.g. from Lemma~\ref{l:versmax} applied to the modular cover $\A_{4,2}\to \A_{4,A_8}$ arising from the diagonal representation of $A_8=\SL_4(\F_2)$; the $\ed$ at 2 of this cover follows from Corollary~\ref{c:gen}).  
\end{para}

\appendix
	\section{On quadratic representations of $S_n$}\label{appendix}
	\begin{center}
			By Nate Harman
	\end{center}
	\subsection{Statement of Results}
	
	Recall that any linear representation of a $p$-group $G$ over a field $k$ of characteristic $p$ contains a non-zero invariant vector, in particular this implies that the only irreducible representation of $G$ over $k$ is the trivial representation. This does not mean that all representations are trivial though, there are non-split extensions of trivial representations and understanding their structure is a central part of modular representation theory.
	
	In a non-semisimple setting, one basic invariant of a representation is its \emph{Lowey length}.  For representations of $p$-groups in characteristic $p$ it can be defined as follows:  Start with a representation $V$ and then quotient it by its space of invariants to obtain a new representation $V' = V / V^G$, then repeat this process until the quotient is zero.  The Lowey length is the number of steps this takes.
	
	In the above work Farb, Kisin, and Wolfson analyze certain special representations of symmetric groups in characteristic $2$, the so-called Dickson embeddings. Typically denoted $D^{(n-1,1)}$ in the representation theory literature,  these representations have the following key property: Let $n = 2m$ or $2m+1$, these representations have Lowey length 2 when restricted to the rank $m$ (which is the maximum possible) elementary abelian 2-subgroup $H_n$ generated by $(1,2), (3,4), \dots, \text{ and } (2m-1,2m)$.  
	
	This motivates the following definition: We say that an irreducible representation of a $S_n$ in characteristic $p$ is \emph{quadratic} with respect to a maximal rank elementary abelian $p$-subgroup $H$ if it has Lowey length 2 upon restriction to $H$.  The purpose of this note is to prove first that this is only a characteristic $2$ phenomenon, and second that these representations $D^{(n-1,1)}$ are the only representations which are quadratic with respect to some maximal rank elementary abelian $p$-subgroup for $n$ sufficiently large ($n\ge 9$).
	
	\medskip
	
	In characteristic $p > 2$, the maximal rank elementary abelian $p$-subgroups in $S_n$ are just those generated by a maximal collection of disjoint $p$-cycles.  Our first main theorem tells us that there are no quadratic representations in characteristic $p >2$, and in fact we can detect the failure to be quadratic here by restricting to a single $p$-cycle.

	\begin{athm} \label{charnot2}
		Any irreducible representation of $S_n$ with $n \ge p$ in characteristic $p>2$ which is not a character has Lowey length at least $3$ upon restriction to the copy of $C_p$ generated by $(1,2,\dots, p)$, and therefore is not quadratic with respect to any maximal rank elementary abelian $p$-subgroup.
	\end{athm}
	
	Note that in any characteristic $p >2$ the characters of $S_n$ are just the trivial and sign representations.
	
	\medskip
	
	In characteristic $2$ things are a bit more complicated. While the subgroup $H_n$ of $S_{n}$ is a maximal rank elementary 2-subgroup, it is no longer the unique such subgroup up to conjugation.  Recall that in $S_4$ there is the Klein four subgroup $K = \{e, (12)(34), (13)(24), (14)(23) \}$, which is a copy of $C_2^2$ not conjugate to $H_4$.
	
	We can construct other maximal rank elementary 2-subgroups of $S_{n}$ by taking products
	
	$$\underbrace{K \times K \times \dots \times K}_{m \text{ times}} \times H_{n-4m} \ \subset \  \underbrace{S_4 \times S_4 \times \dots \times S_4}_{m \text{ times}} \times S_{n-4m} \ \subset \ S_{n}$$
	and up to conjugacy though these are all the maximal rank elementary abelian 2-subgroups inside $S_{n}$.
	
	$S_8$ has a special irreducible representation $D^{(5,3)}$ of dimension $8$ which upon restriction $A_8$ decomposes as a direct sum $D^{(5,3)+} \oplus D^{(5,3)-}$ of two representations of dimension $4$. These representations realize the ``exceptional" isomorphism $A_8 \cong GL_4(\mathbb{F}_2)$, or rather they realize two different isomorphisms differing by either by conjugating $A_8$ by a transposition in $S_8$ or by the inverse-transpose automorphism of $GL_4(\mathbb{F}_2)$.  Under this isomorphism the subgroup $K\times K \subset A_8$ gets identified with the subgroup of matrices of the form
	$$\begin{bmatrix}
	1 & 0 & a & b \\
	0 & 1 & c & d \\
	0 & 0 & 1 & 0 \\
	0 & 0 & 0& 1
	\end{bmatrix}$$
	which is manifestly quadratic. 
	Our second main theorem will be to show that there are no other quadratic representations other than the Dickson embedding once $n$ is at least $9$.
	
	\begin{athm} \label{char2}
		Suppose $V$ is a non-trivial irreducible representation of $S_n$ with $n \ge 9$ over a field of characteristic $2$ which is quadratic with respect to a maximal rank elementary abelian $2$-subgroup $H$. Then $V \cong D^{(n-1,1)}$, and $H$ is conjugate to $H_n$.
	\end{athm}

	\subsection{Proofs of Main Theorems}
		
		We will be assuming a familiarity with the modular representation theory of symmetric groups. A standard reference for this material the book \cite{JamesBook} of James, which we will be adopting the notation from and referring to for all the basic results we need.  The irreducible representations of $S_n$ in characteristic $p$ are denoted by $D^\lambda$, for $p$-regular partitions $\lambda$ of $n$.  These arise as quotients of the corresponding Specht modules $S^\lambda$, which are well behaved reductions of the ordinary irreducible representations in characteristic zero.
		
		\subsubsection{Proof of Theorem \ref{charnot2}}
		
		First we will reduce the problem to just looking at representations of $S_p$. For that we have the following lemma:
		
		\begin{alemma}\label{reduction} \mbox{}
			\begin{enumerate}
				
				\item Every irreducible representation $V$ of $S_n$ with $n \ge p$ in characteristic $p > 3$ which is not a character has a composition factor when restricted to $S_p$ which is not a character.
				
				\item Every irreducible representation $V$ of $S_n$ with $n \ge 4 $ in characteristic $3$ which is not a character has a composition factor when restricted to $S_4$ which is not a character.
				
			\end{enumerate}
		\end{alemma}
		
		\noindent \textbf{Proof:} For part (a) suppose $V$ only has composition factors which are characters when restricted to $S_p$.  If we restrict this to the alternating group $A_p$ all the composition factors must be trivial, as $A_p$ only has the trivial character. If we further restrict to $A_{p-1}$ the whole action must be trivial because representations of $A_{p-1}$ are semisimple in characteristic $p$. However if the action of $A_{p-1}$ is trivial on $V$ then so is the action of the entire normal subgroup generated by $A_{p-1}$ inside $S_n$, which we know is all of $A_n$ if $n>3$. So $V$ must be the trivial as a representation of $A_n$, and is therefore a character of $S_n$.
		
		For part (b), let's again suppose $V$ only has composition factors that are characters when restricted to $S_4$, which implies it only has trivial composition factors when restricted to $A_4$.  If we further restrict to the Klein four subgroup $K$ the whole action must be trivial because representations of $K$ are semisimple in characteristic $p \ne 2$.  As before we see $V$ must be trivial for the normal subgroup of $S_n$ generated by $K$, which we know is all of $A_n$ for $n > 4$. Therefore $V$ is a character.  $\square$
		
		\medskip
		
		\noindent \textbf{Remark:} The modification for characteristic $3$ is necessary because in characteristic $3$ the only irreducible representations of $S_3$ are the trivial and sign representations. Theorem \ref{charnot2} holds vacuously in this case.
		
		\medskip

		It is now enough to prove Theorem \ref{charnot2} for $S_p$ in characteristic $p >3$, and for $S_4$ in characteristic $3$. Let's first focus on the case where $p>3$. If $\lambda$ is a $p$-core, then Nakayama's conjecture (which is actually a theorem, see \cite{JamesBook} Theorem 21.11) tells us $D^\lambda = S^\lambda$ is projective, and hence remains projective when restricted to $C_p$ and therefore has Lowey length $p$.  This leaves those irreducible representations corresponding to hook partitions $ \lambda = (p-k, 1^k)$.
		
		In the simplest case where $\lambda = (p-1,1)$ then $D^\lambda$ is the $(p-2)$-dimensional quotient of the standard $(p-1)$-dimensional representation $S^{(p-1,1)}$ by its one dimensional space of invariants, and one can easily verify this forms a single $(p-2)$-dimensional indecomposable representation of $C_p$. Peel explicitly computed the decomposition matrices for $S_p$ in characteristic $p$ (see \cite{JamesBook} Theorem 24.1), and it follows from his calculation that the remaining irreducible representations $D^\lambda$ with $\lambda = (p-k, 1^k)$ for $1 < k \le p-2$ are just exterior powers $\Lambda^k D^{(p-1,1)}$ of this $(p-2)$-dimensional representation.
		
		Since $k < p$ we know that $\Lambda^k D^{(p-1,1)}$ is a direct summand of $(D^{(p-1,1)})^{\otimes k}$, which as a representation of $C_p$ is just the unique $(p-2)$-dimensional indecomposable representation tensored with itself $k$ times.  Tensor product decompositions for representations of cyclic groups are known explicitly (\cite{Green} Theorem 3), and in particular it is known that a tensor product of two odd dimensional indecomposable representations of $C_p$ always decomposes as a direct sum of odd dimensional indecomposable representations.  So we see $(D^{(p-1,1)})^{\otimes k}$ and $\Lambda^k D^{(p-1,1)} = D^{(p-k, 1^k)}$ only have odd length indecomposable factors when restricted to $C_p$.  If it had Lowey length 1 when restricted to $C_p$ that means the action is trivial, which implies the action of $A_p$ must also be trivial as $A_p$ is simple, but that would imply the original representation of $S_n$ was a character.
		
		In the characteristic $3$ case there are only two irreducible representations of $S_4$, they are the standard 3-dimensional representation $S^{(3,1)} = D^{(3,1)}$ and its sign twisted version $S^{(2,1,1)} = D^{(2,1,1)}$.  These are $3$-core partitions so again by Nakayama's conjecture they are both projective and therefore remain projective when restricted to $C_3$ and have Lowey length 3.
		$\square$
		
		\subsubsection{Proof of Theorem \ref{char2}}
		
		The overall structure of the proof will be to successively rule different classes of representations and maximal rank elementary abelian $2$-subgroups through a sequence of lemmas.  The first such lemma will let us rule out those irreducible representations $D^\lambda$ where $\lambda$ is a $2$-regular partition with at least 3 parts.
		
		\begin{alemma} 
			
			If $\lambda$ is a $2$-regular partition with at least $3$ parts, then the irreducible representation $D^\lambda$ of $S_n$ contains a projective summand when restricted to $S_6$.
		\end{alemma}
		
		\noindent \textbf{Proof:}  Note that any $2$-regular partition $\lambda$ with at least $3$ parts can be written as $(3,2,1) + \mu = (\mu_1 +3, \mu_2 + 2 , \mu_3 + 1, \mu_4, \dots, \mu_\ell)$ for some partition $\mu = (\mu_1, \mu_2, \dots, \mu_\ell)$.  James and Peel \cite{JP} constructed explicit Specht filtrations for $Ind_{S_6 \times S_{n-6}}^{S_n}(S^{(3,2,1)} \otimes S^\mu)$, which have $S^\lambda$ as the top filtered quotient.  In particular this implies $Ind_{S_6 \times S_{n-6}}^{S_n}(S^{(3,2,1)} \otimes S^\mu)$  has $D^\lambda$ as a quotient. However by Frobenius reciprocity we know that
		$$\text{Hom}_{S_n}(Ind_{S_6 \times S_{n-6}}^{S_n}(S^{(3,2,1)} \otimes S^\mu), D^\lambda) \cong \text{Hom}_{S_6 \times  S_{n-6}}(S^{(3,2,1)} \otimes S^\mu, Res_{S_6 \times S_{n-6}}^{S_n}(D^\lambda)).$$
		So since the left hand side is nonzero, the right hand side is as well. 
		
		Now if we look at $S^{(3,2,1)} \otimes S^\mu$ as a representation of $S_6$ it is just a direct sum of $\text{dim} (S^\mu)$ copies of $S^{(3,2,1)}$, which we know is irreducible and projective by Nakayama's conjecture.  In particular the image under any nonzero homomorphism is also just a direct sum of copies of $S^{(3,2,1)}$, so $D^\lambda$ must contain at least one copy of $S^{(3,2,1)}$ as a direct summand. $\square$
		
		\begin{acor} \label{length2}
			If $\lambda$ is a $2$-regular partition with at least $3$ parts, then $D^\lambda$ is not quadratic with respect to any maximal rank elementary abelian $2$-subgroup of $S_n$.
		\end{acor}
		
		\noindent \textbf{Proof:} After conjugating we may assume that our maximal rank elementary abelian subgroup intersects $S_6$ in an elementary abelian 2-group of rank at least $2$.  The previous lemma says any such irreducible representation must contain projective summand when restricted to $S_6$, and then this summand remains projective upon restriction to the intersection of $S_6$ with our maximal rank elementary $2$-subgroup. Projective representations of $C_2^2$ have Lowey length 3, so the Lowey length for the entire maximal rank elementary abelian $2$-subgroup of $S_n$ must be at least that big. $\square$

		\medskip
		
		This reduces the problem to understanding what happens for two-part partitions $\lambda = (a,b)$.  These representations are much better understood then the general case.  For one thing, the branching rules for restriction are completely known in this case, although we'll just need the following simplified version:
		
		\begin{alemma} \label{branching} (See \cite{kleshchev} Theorem 3.6, following \cite{Sheth})  If $(a,b)$ is a two-part partition of $n$ with $a-b > 1$ then $D^{(a-1,b)}$ appears as a subquotient with multiplicity one inside the restriction of $D^{(a,b)}$ to $S_{n-1}$, and the other composition factors are all of the form $D^{(a-1+r,b-r)}$ with $r>0$.
			
		\end{alemma}
		\medskip
		
		Recall that we defined $H_{2k} \subset S_{2k}$ to be the elementary abelian 2-subgroup of $S_{2k}$ generated by the odd position adjacent transpositions $(2i-1, 2i)$ for $1 \le i \le k$,  we will also consider $H_{2k}$ as a subgroup of $S_n$ for $n>2k$ via the standard inclusions of $S_{2k}$ into $S_n$.  The next lemma will be to settle for us exactly which representations are have Lowey length $2$ when restricted to the standard maximal rank elementary abelian subgroup $H_n$.
		
		\medskip
		
		\begin{alemma} \label{H2kproj}
			$D^{(n-k,k)}$ contains a projective summand when restricted to $H_{2k}$.
		\end{alemma}
		
		\noindent \textbf{Proof:} We know from the branching rules (Lemma \ref{branching}) that $D^{(n-k,k)}$ contains a copy of $D^{(k+1,k)}$ as a subquotient when restricted to $S_{2k+1}$, so it is enough to verify it for $D^{(k+1,k)}$.  Moreover $H_{2k} \subset S_{2k}$ so really this calculation is taking place inside $M(2k) := \text{Res}^{S_{2k+1}}_{S_{2k}}D^{(k+1,k)}$. 
		
		These representations $D^{(k+1,k)}$ and $M(2k)$ are well studied. Benson proved $D^{(k+1,k)}$ is a reduction modulo $2$ of the so-called basic spin representation of $S_{2k+1}$ in characteristic zero (\cite{BensonSpin} Theorem 5.1).  Nagai and Uno (\cite{Uno} Theorem 2, or see \cite{Okuyama} Proposition 3.1 for an account in English), gave explicit matrix presentations for $M(2k)$ and showed that they have the following recursive structure:
		
		$$ \text{Res}^{S_{2m}}_{S_{2i} \times S_{2m-2i}} M(2m) \cong M(2i)\otimes M(2m-2i) $$
		
		In particular since $M(2)$ can easily be seen to be the regular representation of $S_2 = H_2$, it follows by induction that $M(2k)$ is projective (and just a single copy of the regular representation) for $H_{2k}$.  $\square$

		\begin{acor}  The only nontrivial irreducible representation of $S_n$ which is quadratic with respect to $H_n$ is $D^{(n-1,1)}$.
			
		\end{acor}
		
		\noindent \textbf{Proof:} Corollary \ref{length2} tells us that if $\lambda$ has at least $3$ parts, $D^\lambda$ has Lowey length at least 3 when restricted to $H_n$.  Then Lemma \ref{H2kproj} tells us that $D^{(n-k,k)}$ has Lowey length at least $k+1$ as a $H_n$ representation and is therefore not quadratic for $k>1$.  $\square$
		
		\medskip
		
		To finish the proof of Theorem \ref{char2} we need to show that for $n$ at least $9$ that there are no representations which are quadratic with respect to to any of these other maximal rank elementary abelian $2$-subgroups $K^m \times H_{n-4k}$ with $m \ge 1$. Lemma \ref{length2} rules out $D^\lambda$ for $\lambda$ of length at least $3$, so again we will just need to address the case when $\lambda$ is a length 2 partition.
		
		We do this through a series of lemmas ruling out different cases, but first will state the following well-known fact from the modular representation theory of symmetric groups:
		
		\begin{alemma} \label{spechtbranch} (\cite{JamesBook} Theorem 9.3) If $\lambda$ is a partition of $n$, then $S^\lambda$ restricted to $S_{n-1}$ admits a filtration $$0=M_0 \subset M_1 \subset \dots \subset M_N \cong S^\lambda$$ such that the successive quotients $M_i / M_{i-1}$ are isomorphic to Specht modules $S_\mu$, and $S^\mu$ appears if and only if $\mu$ is obtained from $\lambda$ by removing a single box, in which case it appears with multiplicity one.
		\end{alemma}

		\begin{alemma}
			$D^{(n-1,1)}$, for $n \ge 5$, contains a projective summand when restricted to $K$, and is therefore not quadratic with respect to any group containing $K$.
			
		\end{alemma}
		
		\noindent \textbf{Proof:} It suffices to prove it for $D^{(4,1)}$ as every $D^{(n-1,1)}$ for $n >5$ contains it as a composition factor upon restriction to $S_5$ by Lemma \ref{branching}.   This representation $D^{(4,1)}$ is just the $4$ dimensional subspace of $\mathbb{F}_2^5$ where the sum of the coordinates is zero.  If we restrict this representation to $S_4$ this can be identified with the standard $4$-dimensional permutation representation via the map $(a,b,c,d) \to (a,b,c,d,-a-b-c-d)$.  The restriction of the standard action of $S_4$ on a $4$-element set to $K$ is simply transitive, so this representation is just a copy of the regular representation. $\square$
		
		\medskip 
		
		\begin{alemma}
			$D^{(n-2,2)}$ for $n \ge 7$ and $D^{(n-3,3)}$ for $n \ge 9$ each contain a projective summand when restricted to $K$, and are therefore not quadratic with respect to any group containing $K$.
			
		\end{alemma}
		
		\noindent \textbf{Proof:} It suffices to prove it for $D^{(5,2)}$ and $D^{(6,3)}$ as every $D^{(n-2,2)}$ for $n >7$ contains $D^{(5,2)}$ as a composition factor upon restriction to $S_7$, and similarly every $D^{(n-2,2)}$ for $n >7$ contains $D^{(6,3)}$ as a composition factor upon restriction to $S_9$ by Lemma \ref{branching}.
		
		For $S_7$ and $S_9$ the decomposition matrices are known explicitly and we have that $D^{(5,2)} = S^{(5,2)}$ and $D^{(6,3)} = S^{(6,3)}$ (see the appendix of \cite{JamesBook}). For Specht modules the branching rules are given by Lemma \ref{spechtbranch}  and $S^{(5,2)}$ and $S^{(6,3)}$ both contain $S^{(4,1)}$ as a subquotient upon restriction to $S_5$. The result then follows from the previous lemma. $\square$
		
		\medskip
		
		\begin{alemma}
			$D^{(n-k,k)}$ for $k \ge 4$ and $n \ge 2k+1$ is not quadratic when restricted to $K^m \times H_{n-4m}$ for any $m \ge 1$.
			
		\end{alemma}
		
		\noindent \textbf{Proof:} We know by Lemma \ref{H2kproj} these are projective upon restriction to $H_{2k}$, and are therefore projective when restricted to the intersection of $H_{2k}$ with $K^m \times H_{n-4m}$. This intersection has rank at least 2 since $k\ge 4$, and therefore projective objects have Lowey length at least 3.   $\square$
		
		\medskip
		
		\noindent This completes the proof of Theorem \ref{char2}. $\square$

	\subsection{Modifications for $A_n$}
	
	We will now briefly describe what changes if we work with alternating groups instead of symmetric groups, but we will omit some of the details of the calculations.  First a quick summary of the modular representation theory of alternating groups in terms of the theory for symmetric groups:
	
	Upon restriction from $S_n$ to $A_n$, the irreducible representations $D^\lambda$ either remain irreducible, or split as a direct sums $D^\lambda \cong D^{\lambda+} \oplus D^{\lambda-}$ of two irreducible non-isomorphic representations of the same dimension; all irreducible representations of $A_n$ are uniquely obtained this way.  We'll note that in characteristic $p > 2$ this is a standard application of Clifford theory, but in characteristic $2$ it is a difficult theorem of Benson (\cite{BensonSpin} Theorem 1.1).  Moreover it is known exactly which $D^\lambda$ split this way, but we won't go into the combinatorics here.
	
	When $p >2$ the maximum rank abelian $p$-groups in $S_n$ all lie in $A_n$, and the proof of Theorem \ref{charnot2} goes through without modification to give the following theorem.
	
	\begin{customthm}{\ref{charnot2}'}
		Any non-trivial irreducible representation of $A_n$ with $n \ge p$ in characteristic $p>2$ has Lowey length at least $3$ upon restriction to the copy of $C_p$ generated by $(1,2,\dots, p)$, and is therefore not quadratic with respect to any maximal rank elementary abelian $p$-subgroup.
	\end{customthm}
	
	When $p=2$, the difference is more dramatic. It is no longer true that every maximum rank abelian $2$-subgroup of $S_n$ lies in $A_n$, in particular $H_n$ is not a subgroup of $A_n$. Let $\tilde{H}_n$ denote the intersection of $H_n$ and $A_n$, this has rank one less than $H_n$.  The maximal rank elementary abelian 2-subgroups of $A_n$ are as follows:
	
	If $n = 4b$ or $4b+1$ then up to conjugacy the only maximal rank elementary abelian 2-subgroup inside $A_n$ is $K^b$, and it is of rank $2b$.  If $n = 4b+2$ or $4b+3$ then all maximal rank elementary abelian 2-subgroups in $S_n$ still have maximal rank when intersected with $A_n$, and up to conjugacy the maximal rank elementary abelian 2-subgroups inside $A_n$ are of the form:
	
	$$\underbrace{K \times K \times \dots \times K}_{m \text{ times}} \times \tilde{H}_{n-4m} \ \subset \  \underbrace{A_4 \times A_4 \times \dots \times A_4}_{m \text{ times}} \times A_{n-4m} \ \subset \ A_{n}$$
	and these have rank $2b-1$.
	
	The appropriate modification to Theorem \ref{char2} for alternating groups is the following:
	
	\begin{customthm}{\ref{char2}'}
		Suppose $V$ is a non-trivial irreducible representation of $A_n$ with $n \ge 9$ over a field of characteristic $2$ which is quadratic with respect to a maximal rank elementary abelian $2$-subgroup $H$. Then $n \equiv 2 \text{ or } 3$ modulo $4$, $V \cong D^{(n-1,1)}$, and $H$ is conjugate to $\tilde{H}_n$.
	\end{customthm}
	
	The proof of Theorem \ref{char2} mostly goes through in this case. Some additional care is needed to handle the representations $D^{\lambda+}$ and $D^{\lambda-}$ which are not restrictions of irreducible representations of $S_n$, however one simplifying observation is that since $D^{\lambda+}$ and $D^{\lambda-}$ just differ by conjugation by a transposition, they are actually isomorphic to one another upon restriction to a maximal rank elementary abelian $2$-subgroup. We will omit the remaining details though.

\end{document}